\documentclass[12pt]{amsart}

\usepackage[a4paper,margin=23mm,top=30mm]{geometry}

\usepackage{amsfonts, amsmath, amssymb, amsgen, amsthm, amscd}
\usepackage{newtxtext,newtxmath}
\usepackage[utf8]{inputenc} 

\usepackage{color}

\usepackage[all]{xy}

\usepackage{hyperref}
\usepackage{mathtools,slashed}

\usepackage{enumerate}

\def\a{\alpha}

\def\d{\delta}
\def\D{\Delta}

\def\s{\sigma}

\def\ve{\varepsilon}
\def\vp{\varphi}

\def\Id{\mathop{\rm Id}\nolimits}

\def\lra{\longrightarrow}

\def\ot{\otimes}
\def\odots{\ot\cdots\ot}

\def\lra{\longrightarrow}

\def\D{\Delta}

\def\Id{\mathop{\rm Id}\nolimits}
\def\Im{\mathop{\rm Im}\nolimits}

\newcommand{\ps}[1]{~\hspace{-4pt}_{^{(#1)}}}

\usepackage{enumerate}

\newcommand{\C}[1]{\mathcal{#1}}
\newcommand{\B}[1]{\mathbb{#1}}
\newcommand{\lmod}[1]{{#1}\text{-{\bf Mod}}}
\newcommand{\wbar}[1]{\overline{#1}}

\newcommand{\xra}[1]{\xrightarrow{\ #1\ }}

\newcommand{\Hom}{{\rm Hom}}

\newcommand{\norm}{\text{\bf Vect}^{n}_{\mathbb C}}
\newcommand{\Mix}{\text{\bf Mix}_{\mathbb C}}
\newcommand{\vect}{\text{\bf Vect}_{\mathbb C}}
\newcommand{\Banach}{\text{\bf Ban}_{\mathbb C}}
\newcommand{\cotimes}{\widehat{\otimes}}

\newcommand{\diag}{{\rm Diag}}
\newcommand{\ie}{{\it i.e.\/}\ }

\renewcommand{\leq}{\leqslant}
\renewcommand{\geq}{\geqslant}

\numberwithin{equation}{section}

\newtheorem{theorem}{Theorem}[section]
\newtheorem{proposition}[theorem]{Proposition}
\newtheorem{lemma}[theorem]{Lemma}

\theoremstyle{definition}

\newtheorem{remark}[theorem]{Remark}

\newtheorem{definition}[theorem]{Definition}

\def \vb{\uparrow\hspace{-1pt}b}
\def \vB{\uparrow\hspace{-1pt}B}

\def\hb{\overset{\rightarrow}{b}}
\def\hB{\overset{\rightarrow}{B}}

\title{The asymptotic Connes-Moscovici characteristic map and the index cocycles}

\author{A. Kaygun}
\address{Istanbul Technical University, Istanbul, Turkey}
\email{kaygun@itu.edu.tr}

\author{S. Sütlü}
\address{I\c{s}ik University, Istanbul, Turkey}
\email{serkan.sutlu@isikun.edu.tr}

\begin{document}

\begin{abstract}
  We show that the (even and odd) index cocycles for theta-summable Fredholm modules
  are in the image of the Connes-Moscovici characteristic map. To show
  this, we first define a new range of asymptotic cohomologies, and
  then we extend the Connes-Moscovici characteristic map to our setting.
  The ordinary periodic cyclic cohomology and the entire cyclic cohomology
  appear as two instances of this setup.  We then construct an asymptotic
  characteristic class, defined independently from the underlying
  Fredholm module. Paired with the $K$-theory, the image of this class under the characteristic map yields a
  non-zero scalar multiple of the index in the even case, and the spectral flow in the odd case.
\end{abstract}

\maketitle

\section*{Introduction}

We continue investigating the Connes-Moscovici
characteristic map~\cite{ConnMosc98,ConnMosc99} in relation to more
geometric invariants of noncommutative spaces.  This time, our target is
the entire cyclic cohomology of Banach algebras, or more generally of mixed
complexes and (co)cyclic objects in the category of Banach
spaces~\cite{Conn88}.  

Using Connes' fundamental idea in defining the entire cyclic
cohomology~\cite{Conn88}, we observe a new range of cyclic
cohomologies whose cocycles determined by their asymptotic growth.
The ordinary periodic and entire cyclic cohomologies are but two specific
examples of these asymptotic invariants.  Then we observe that the
non-trivial asymptotic invariants of noncommutative spaces different from the ordinary periodic
cyclic cohomology are beyond the reach of cohomological
$\delta$-functors~\cite[Sect. 2.1]{Weib-book}, or more appropriately,
cohomological functors defined on the triangulated category of mixed
complexes.  We refer reader to Proposition~\ref{Collapse} for the exact statement.
This fact has two important consequences. Firstly, asymptotic
cohomologies are finer invariants.  Secondly, these asymptotic
invariants are different from Higson's E-theory~\cite{Higson:KKTheory}
or Puschnigg's asymptotic
cohomology~\cite{Puschnigg:AsymptoticCohomology}, since both of these
invariants come from specific cohomological bifunctors on two
different derived categories of Banach spaces using the Connes-Higson
asymptotic morphisms~\cite{ConnesHigson:AsymptoticMorphisms}.
Curiously, we also observe that if one were to construct similar
asymptotic invariants on Tsygan's cyclic bicomplex, one would obtain
drastically different results.  We refer the reader to
Proposition~\ref{TrivialityOnBounded} for the details.

Entire cyclic cohomology is crucially useful for the Chern character
of the theta-summable Fredholm
modules~\cite{JaffLesnOste88,GetzSzen89,Conn88}.  Our main results in
this paper are Theorem~\ref{EvenIndexTheorem} and Theorem \ref{thm-spect-flow}, where we show that the
even and the odd index cocycles of theta-summable Fredholm modules are in the image
of an asymptotic analogue of the Connes-Moscovici characteristic
map.  More explicitly, in Proposition~\ref{UniversalIndexCocycle} we show that there is an asymptotic class,
defined independently from the underlying Fredholm module, whose image under a 
Connes-Moscovici type characteristic homomorphism pairs with $K_0$ (resp. $K_1$) 
and yields a non-zero scalar multiple of the index (resp. the spectral flow).


\subsection*{The plan of the article}

In Section~\ref{Preliminaries} we set the notation and list the basic
objects and tools we are going to need in the sequel.  In
Section~\ref{Asymptotic} we define a range of asymptotic cohomologies, an example of which is
the entire cyclic cohomology.  In Section~\ref{CohomologicalFunctors} we investigate what 
happens if asymptotic cohomology were to come from a cohomological functor in two 
cases: first for asymptotic cohomology of mixed complexes, and then for 
asymptotic cohomology of cyclic bicomplexes.  We observe that the results diverge
dramatically.  The proper context for Connes-Moscovici characteristic map is an 
appropriate version of the cup product in cohomology.  So, we develop such a cup 
product for asymptotic cohomologies in Section~\ref{Products}.  
In Section~\ref{AsymptoticSimplex} we construct an asymptotic complex out of the geometric 
$n$-simplicies, for $n\geq 0$, and show that this complex contains a non-trivial cocycle.  
Finally, in Section~\ref{ChernConnesCharacter} we construct an asymptotic analogue of the Connes-Moscovici 
characteristic homomorphism explicitly for both the even and the odd theta-summable Fredholm modules. 
We then apply our machinery to the asymptotic class of Section~\ref{AsymptoticSimplex} to 
obtain the index in the even case, and the spectral flow in the odd case.

\subsection*{Notation and Conventions}

Throughout the article, we work with complete normed vector spaces
over the field of complex numbers $\B{C}$, and the completed tensor
product $\cotimes$ over $\B{C}$.  Since we crucially use norm
estimates and their growth, the results of this paper do not
immediately generalize to Fréchet spaces, or to more general
topological spaces.  All (co)cyclic modules are \emph{normalized} in
the sense that the operator norms of the (co)face and (co)degeneracy
operators grow at most linearly with respect to their simplicial
degree, i.e. they are of asymptotic order $\C{O}(n)$ where $n$
indicates the simplicial degree.  The mixed complex of a (co)cyclic
module $\C{C}^\bullet$ is referred as the $(b,B)$-complex of
$\C{C}^\bullet$ while Tsygan's bicomplex is simply referred as the
cyclic bicomplex of $\C{C}^\bullet$.


\section{Preliminaries}\label{Preliminaries}

\subsection{The monoidal category of normed (complete) vector spaces}\label{BaseCategory}

Consider the category of (small) normed vector spaces $\norm$ over
$\B{C}$.  Each object $(V,\|\cdot\|)$ consists of a (small) vector
space $V$ endowed with a norm $\|\cdot\|$.  The underlying metric
spaces are not expected to be complete.  

There is a full and faithful subcategory $\Banach$ of $\norm$ that
consists of \emph{complete} normed vector spaces.

The usual algebraic tensor product $\otimes$ over $\B{C}$ gives a
monoidal product on $\norm$.  For a given pair $(V,\|\cdot\|_V)$ and
$(W,\|\cdot\|_W)$ of normed vector spaces, we endow $V\otimes W$ with
the norm
\[ \|u\| = \inf\left\{\sum_i \|v_i\|_V\|w_i\|_W:\ u = \sum_i
    v_i\otimes w_i\right\}. \] Notice that $\B{C}$, with its standard
complex norm, is the unit object in this category.  

We note also that since the product of two normed vector spaces has a
norm, but is not required to be complete, the subcategory $\Banach$ is
not a monoidal subcategory. However, one can complete such a product
$V\otimes W$ to obtain a complete normed vector space denoted by
$V\cotimes W$.  This construction defines a monoidal product on
$\Banach$ which is different than that of $\norm$.  For details,
see for instance \cite[Chapt. 2]{Ryan-book}.

In this paper we are going to work with the monoidal category
$(\Banach,\cotimes)$. We can
define algebras, coalgebras, bialgebras and Hopf algebras in the
strict monoidal category $(\Banach,\cotimes)$.  All such
objects are assumed to be norm complete and the structure maps (multiplications,
comultiplications, units, counits and antipodes) are all bounded
(continuous).

\subsection{Simplicial and (para)(co)cyclic objects}

Let $\Delta$, $\Delta C$ and $\Delta \B{Z}$ respectively be the
simplicial, the cyclic and the paracyclic categories \cite{Connes-book,Loday-book}.  A
simplicial object $\C{S}_\bullet$ is a functor of the form
$\C{S}_\bullet\colon \Delta^{op}\to \Banach$, and a cyclic object
$\C{C}_\bullet$ in $\Banach$ is a similar functor of the form
$\C{C}_\bullet\colon \Delta C^{op}\to\Banach$.  Paracyclic objects are
defined similarly to be the functors of the form
$\Delta\B{Z}^{op}\to \Banach$.  The ``co'' versions are obtained by
replacing $\Delta^{op}$, $\Delta C^{op}$ and $\Delta\B{Z}^{op}$ with
their categorical duals $\Delta$, $\Delta C$ and $\Delta\B{Z}$,
respectively.  As such, again, all structure maps are bounded.

\subsection{Mixed complexes}

A $\B{N}$-graded vector space $X_*$ is called \emph{a mixed complex}
if it is equipped with two different differentials $b_n\colon X_n\to X_{n+1}$ and
$B_n\colon X_n\to X_{n-1}$ such that for every $n\in \B{N}$,
\[ b_{n+1}b_n = 0, \qquad B_{n-1}B_n = 0,\quad  \text{ and }\quad b_{n-1}B_n +
B_{n+1}b_n = 0.\]  

Given two mixed complexes $(X_*,b_*,B_*)$ and $(X'_*,b'_*,B'_*)$, a
graded morphism $f_*\colon X_*\to X'_*$ is called a morphism of mixed
complexes if $f_*$ yields two morphisms of differential graded modules
of the form $f_*\colon (X_*,b_*)\to (X'_*,b'_*)$ and
$f_*\colon (X_*,B_*)\to (X'_*,B'_*)$.  The (small) category of mixed
complexes is denoted by $\Mix$.  The category $\Mix$ is a $\B{C}$-linear abelian exact
category~\cite{Quillen:KTheory}.

\subsection{Mixed complex of a cocyclic module}

Let $\C{C}^\bullet$ be a cocyclic object in $\vect$, with the cocyclic structure given by 
\begin{align}\label{GenericCocyclicStructure}
  d_i:\C{C}^n \lra & \C{C}^{n+1}, \qquad 0\leq i \leq n+1, \nonumber\\
  s_j:\C{C}^n \lra & \C{C}^{n-1}, \qquad 0\leq j \leq n-1, \\
  t_n:\C{C}^n \lra & \C{C}^n.\nonumber
\end{align}
Then, $(\C{C}^\bullet,b_\ast,B_\ast)$ with two differentials
\begin{equation*}
  b_n:= \sum_{i=0}^{n+1}(-1)^id^n_i \quad\text{ and }\quad B_n:=N_{n-1} s^n_{n-1}t_n(\Id - (-1)^nt_n), \quad\text{ where }\quad N_{n-1}:=\sum_{i=0}^{n-1}\,(-1)^{i(n-1)}t_{n-1}^i
\end{equation*}
is a mixed complex.

The following proposition is well-known.  We refer the reader
to~\cite{Kassel:1987} for a proof.

\begin{proposition}
  Given a cocyclic module $\C{C}^\bullet$, the assignment
  $\C{C}^\bullet\mapsto \C{M}(\C{C}^\bullet):=(\C{C}^*,b_*,B_*)$ defines an exact functor
  from $\lmod{\Delta C}$, the category of cocyclic modules, to $\Mix$,
  the category of mixed complexes.
\end{proposition}

\subsection{Periodic cyclic cohomology of cocyclic objects}\label{PeriodicCyclicComplex}

Let $\C{C}^\bullet$ be a cocyclic object in $\vect$, whose cocyclic structure is given as in \eqref{GenericCocyclicStructure}.
The periodic cyclic cohomology of $\C{C}^\bullet$
is the direct sum total homology of the cyclic bicomplex, \ie the upper half plane bicomplex $\C{CC}^{p,q} :=
\C{C}^{q}$ with $p,q\in\B{Z}$, $q \geq 0$, 
\begin{equation}\label{bicomplex-Tsygan}
\xymatrix{
& \vdots & \vdots & \vdots & \vdots\\
\cdots\ar[r]^N & \C{C}^2 \ar[u]^b \ar[r]^{1-t} & \C{C}^2 \ar[u]^{b'} \ar[r]^{N} & \C{C}^2 \ar[u]^b \ar[r]^{1-t} & \C{C}^2 \ar[u]^{b'} \ar[r]^{N} & \cdots\\
\cdots\ar[r]^N & \C{C}^1 \ar[u]^b \ar[r]^{1-t} & \C{C}^1 \ar[u]^{b'} \ar[r]^{N} & \C{C}^1 \ar[u]^b \ar[r]^{1-t} & \C{C}^1 \ar[u]^{b'} \ar[r]^{N} & \cdots\\
\cdots\ar[r]^N & \C{C}^0 \ar[u]^b \ar[r]^{1-t} & \C{C}^0 \ar[u]^{b'} \ar[r]^{N} & \C{C}^0 \ar[u]^b \ar[r]^{1-t} & \C{C}^0 \ar[u]^{b'} \ar[r]^{N} & \cdots}
\end{equation}
where 
\begin{equation*}
b:= \sum_{i=0}^{n+1}(-1)^id^n_i, \qquad b':= \sum_{i=0}^n(-1)^id^n_i, \qquad N_n:= \sum_{i=0}^n(-1)^{ni}t_n^i,
\end{equation*}
In other words,
\begin{equation*}
HP^*(\C{C}^\bullet) := H_\ast({\rm Tot}_\oplus\,\C{CC}^{\bullet,\bullet})
\end{equation*}
is given by the periodic differential graded complex
\begin{equation}\label{cochain-groups}
 {\rm Tot}_\oplus^i\,\C{CC}^{\bullet,\bullet} := \bigoplus_{n\geq 0} \C{C}^{n} \quad\text{with}\quad \partial_i\colon {\rm Tot}_\oplus^i\,\C{CC}^{\bullet,\bullet} \to {\rm Tot}_\oplus^{1-i}\,\C{CC}^{\bullet,\bullet}
\end{equation}
for $i\in \{0,1\}$, which are called the even and the odd cochain groups respectively. 

Allowing the direct products in \eqref{cochain-groups} rather than the
direct sums, one defines the spaces
\begin{equation}\label{inf-support}
  {\rm Tot}_\infty^i\,\C{CC}^{\bullet,\bullet} := \prod_{n\geq 0} \C{C}^n, \qquad i\in \{0,1\}
\end{equation}
of even and odd cochains with infinite support, and hence the direct product total 
\begin{align*}
\xymatrix{
\ldots \ar[r] & {\rm Tot}_\infty^0\,\C{CC}^{\bullet,\bullet} \ar[r]^{\partial_0} & {\rm Tot}_\infty^1\,\C{CC}^{\bullet,\bullet} \ar[r]^{\partial_1} & {\rm Tot}_\infty^0\,\C{CC}^{\bullet,\bullet} \ar[r]^{\partial_0} & {\rm Tot}_\infty^1\,\C{CC}^{\bullet,\bullet} \ar[r]^{\,\,\,\,\,\,\,\,\partial_1} & \ldots
}
\end{align*}
The homology of the latter is called the periodic cyclic cohomology with infinite support 
\begin{equation*}
HP_\infty^*(\C{C}^\bullet) := H_\ast({\rm Tot}_\infty\,\C{CC}^{\bullet,\bullet}).
\end{equation*}
We note from
\cite{Conn88,Khalkhali-thesis} that this homology is trivial for every
cocyclic module $\C{C}^\bullet$.

\subsection{Periodic cyclic cohomology of mixed complexes}

Let $(\C{C}^*,b_\ast,B_\ast)$ be a mixed complex.  We can construct a
bicomplex $\C{CC}^{p,q} := \C{C}^{q+p}$ with $p,q\in\B{Z}$ such that
$q \geq -p$, given by
\begin{equation}\label{bicomplex-bB}
\xymatrix{
& \vdots & \vdots & \vdots & \vdots\\
\cdots\ar[r]^B & \C{C}^3 \ar[u]^b \ar[r]^B & \C{C}^2 \ar[u]^b \ar[r]^B & \C{C}^1 \ar[u]^b \ar[r]^B & \C{C}^0 \ar[u]^b & \\
\cdots\ar[r]^B & \C{C}^2 \ar[u]^b \ar[r]^b & \C{C}^1 \ar[u]^b \ar[r]^B & \C{C}^0 \ar[u]^b  &  & \\
\cdots\ar[r]^B & \C{C}^1 \ar[u]^b \ar[r]^B & \C{C}^0  \ar[u]^b &  &  & \\
\cdots\ar[r]^B & \C{C}^0\ar[u]^b &  &  &  &  }
\end{equation}
Similar to the bicomplexes we defined in
Subsection~\ref{PeriodicCyclicComplex} we again have
\[ {\rm Tot}_\oplus^i\,\C{CC}^{\bullet,\bullet} := \bigoplus_{n\geq 0} \C{C}^{n} \quad\text{ and }\quad
   {\rm Tot}_\infty^i\,\C{CC}^{\bullet,\bullet} := \prod_{n\geq 0} \C{C}^n, \qquad i\in \{0,1\}
\]
together with the differentials $\partial^i = b+B$ with $i=0,1$.  These
complexes give us $HP^*(\C{C}^*)$ and $HP_\infty^*(\C{C}^*)$,
respectively.

In case $\C{C}^\bullet$ is a cocyclic object such that the
corresponding bar complex $(\C{C}^*,b')$ is acyclic, then the double
complex given in \eqref{bicomplex-Tsygan} and the double
complex given in \eqref{bicomplex-bB} are quasi-isomorphic in
both direct sum and direct product versions. See~\cite{Khalkhali-thesis} for
details.

\section{Asymptotic hierarchy and cohomology}\label{Asymptotic}

As is observed in \cite{Conn88,JaffLesnOste88,Khalkhali-thesis,KlimLesn93}, the cohomology of the subcomplex of cochains of infinite support satisfying certain growth conditions might be nontrivial. In the following section, we are going to consider the subcomplexes of cochains with specific growth conditions.

\subsection{Asymptotic hierarchy of sequences}

Given two sequences $(x_n)_{n\geq 0}$ and $(y_n)_{n\geq 0}$
of real numbers in $(0,\infty)$, we define
\begin{equation*}
  \C{E}(y_n) := \{ (x_n)_{n\geq 0} \mid (x_n)_{n\geq 0}\prec (y_n)_{n\geq 0} \},
\end{equation*}
where we write $(x_n)_{n\geq 0}\prec (y_n)_{n\geq 0}$ if 
\begin{equation}\label{defn:decay}
  \limsup_{n\to \infty} \frac{r^n x_n}{y_n} = 0,
\end{equation}
for every $r\in (0,\infty)$.



We leave the proof of the following simple lemmas to the reader.
\begin{lemma}\label{lemma-entire1}
  Let $(x_n)_{n\geq 0}$ and $(y_n)_{n\geq 0}$ be two sequences in $(0,\infty)$.  Then
  $(x_n)_{n\geq 0}\prec (y_n)_{n\geq 0}$ if and only if
  \begin{equation}\label{eq:RootTest}
    \limsup_{n\to \infty} \sqrt[n]{\frac{x_n}{y_n}} = 0.
  \end{equation}
\end{lemma}

  

\begin{lemma}\label{lemma-entire2}
  Let $(x_n)_{n\geq 0}$ and $(y_n)_{n\geq 0}$ be two sequences
  in $(0,\infty)$. Then the radius of convergence of
  the power series $\sum_{n\geq 0} \frac{x_n}{y_n} z^n$ is infinite if
  and only if $(x_n)\prec (y_n)$.
\end{lemma}


\begin{remark}
  The asymptotic decay condition \eqref{defn:decay} is quite strong.
  For example, $\C{E}(\lambda^n) = \C{E}(1)$ and
  $\C{E}(n^a) = \C{E}(1)$ for every $a,\lambda\in[0,\infty)$. 
\end{remark}

\subsection{Asymptotic cohomology of cocyclic objects}\label{sect:move}

Assume $\C{C}^\bullet$ is a cocyclic object in $\Banach$.  We define a
graded subspace of $C^*_\infty(\C{C}^\bullet)$ by letting
\begin{equation}\label{E-complex}
C^i_{\C{E}(x_n)}(\C{C}^\bullet) := \{(\vp_{2n+i})_{n\geq 0}\in C^i_\infty(\C{C}^\bullet) \mid (\|\vp_{2n+i}\|)_{n\geq 0}\in\C{E}(x_n)\}, \qquad i = 0,1.
\end{equation}
On the other extreme, setting $C^i_\oplus(\C{C}^\bullet) := {\rm
  Tot}_\oplus^i\,\C{CC}^{\bullet,\bullet} \subseteq {\rm
  Tot}_\infty^i\,\C{CC}^{\bullet,\bullet}$
the direct sum total computing the usual (algebraic) periodic cyclic
cohomology of $\C{C}^\bullet$, we get
$C^i_\oplus(\C{C}^\bullet) \subseteq C^i_{\C{E}(x_n)}(\C{C}^\bullet)$
for any $(x_n)_{n\geq 0}$, and $i=0,1$.

Now, given a sequence $(x_n)_{n\geq 0}$, we wish
$C^i_{\C{E}(x_n)}(\C{C}^\bullet)$ to be a subcomplex of
$C^*_\infty(\C{C}^\bullet)$.  Then we need to show that the graded
space \eqref{E-complex} is closed under the differentials $b$ and $B$
which, in turn, requires an estimate for the operator norm of these
operators.

We first note that since the cyclic operators $t_n$ satisfy
$t_n^{n+1} = \Id$, it follows at once that $\|t_n\|^n = 1$, and hence
$\|t_n\|=1$.  Also, since the norm operator is defined as
$N_n=\sum_{i=0}^n(-1)^{ni}t_n^i$, we have $\|N_n\|\leq n+1$.  On the
other hand, if $\C{C}^\bullet$ were the standard cocyclic object
associated to a unital Banach algebra, see for instance
\cite{Conn88,Khalkhali-thesis}, the coface operators
$d_i:\C{C}^n \to \C{C}^{n+1}$ would be given by
$$d_i\vp_n(a_0,\ldots, a_{n+1}) = \vp_n(a_0,\ldots, a_ia_{i+1}, \ldots,
a_{n+1})$$
and hence we would have $\|d_i\vp_n\|=\|\vp_n\|$, for any
$\vp_n \in \C{C}^n$. Since $b_n:=\sum_{i=0}^{n+1}(-1)^id_i$ and
$b'_n:=\sum_{i=0}^n(-1)^id_i$, we would have obtained the estimates
\[ \|b_n\vp_n\|\leq (n+2)\|\vp_n\| \quad\text{ and }\quad
\|b'_n\vp_n\|\leq (n+1)\|\vp_n\| \]
for any $\vp_n \in \C{C}^n$~\cite[Lemma
2.1.2]{Khalkhali-thesis}. However, we do not have these bounds on the
norms of the coface operators for an arbitrary cocyclic module
$\C{C}^\bullet$.  Nevertheless, the coface maps $d^n_i$, the
codegeneracy maps $s^n_j$ and the cyclic maps $t_n$ satisfy the
relations
\begin{equation*}
d^n_i = t^i_{n+1}d^n_0t_n^{-i}, \quad\text{ and }\quad
s^n_j = t_n^{-i}s^n_0t^i_{n+1}
\end{equation*}
since $\C{C}^\bullet$ is a cocyclic object~\cite{Loday-book,
  Connes-book}. As a result, we can control the norms of the coface
and codegeneracy operators by controlling only the norm of the $0$-th
coface and $0$-th codegeneracy operators on each grade. Accordingly,
we give the following definition.

\begin{definition}
  A cocyclic object $\C{C}^\bullet$ with coface maps $d^n_i$ and
  codegeneracy maps $s^n_j$ for $0\leq i\leq n$ and $0\leq j\leq n-1$,
  is called asymptotically normalized if $\|d^n_0\|=1 = \|s^n_0\|$ for any $n\geq 0$.
\end{definition}

Let $(\C{C}^\bullet,d^n_i,s^n_j,t_n)$ be a cocyclic object. Then $(\C{C}^\bullet,\frac{d^n_i}{\|d^n_0\|},\frac{s^n_j}{\|s^n_0\|},t_n)$ satisfies all cocyclic identities except 
\begin{equation*}
\s_j^{n+1}\delta_i^n = \Id, \qquad \text{ for }\,i=j, \,\text{ or } \,i=j+1. 
\end{equation*}
Instead, setting $\delta_i^n:=\frac{d^n_i}{\|d^n_0\|}$ and $\s^n_j:=\frac{s^n_j}{\|s^n_0\|}$, we get
\begin{equation*}
\s_j^{n+1}\delta_i^n = \frac{1}{\|d^n_0\|\|s^n_0\|}\Id, \qquad \text{ for }\,i=j, \,\text{ or } \,i=j+1. 
\end{equation*}
However, the automorphism $\psi_n:\C{C}^n\lra \C{C}^n$ defined as
$\psi_n := \|d^n_0\|\|s^n_0\|\Id$ commutes with all structure maps:
\begin{equation*}
\psi_{n+1}\circ \delta_i^n = \delta_i^n\circ\psi_n, \qquad \psi_{n-1}\circ \s_j^n = \s_j^n\circ\psi_n, \qquad \psi_{n-1}\circ t_n = t_n\circ\psi_n.
\end{equation*}
Hence, the associated mixed complex
$$(\C{C}^\bullet,\, b:=\sum_{i=0}^{n+1}(-1)^id^n_i,\, B:=N_{n-1} s^n_{n-1}t_n(\Id - (-1)^nt_n))$$ 
is isomorphic to the mixed complex
$$(\C{C}^\bullet,\,\widetilde{b}:=\sum_{i=0}^{n+1}(-1)^i\delta^n_i,\, \widetilde{B}:=N_{n-1} \s^n_{n-1}t_n(\Id -
(-1)^n t_n)).$$
As a result, we assume that all (co)cyclic modules are asymptotically normalized from this point on.

\begin{proposition}\label{prop-subcomplex}
  Given a cocyclic object $\C{C}^\bullet$ in $\Banach$, and a sequence
  $(x_n)_{n\geq 0}$ in $(0,\infty)$, the graded space
  $C^\ast_{\C{E}(x_n)}(\C{C}^\bullet)$ is a differential graded
  subspace of the periodic cyclic complex with infinite support
  $C^\ast_{\infty}(\C{C}^\bullet)$.  
\end{proposition}

\begin{proof}
  Let $\vp = (\vp_{2n+i})_{n\geq 0}\in C^i_{\C{E}(x_n)}$ \ie $(\|\vp_{2n+i}\|)_{n\geq 0}\in \C{E}(x_n)$.  Then for any $r \in (0,\infty)$,
  \begin{align*}
    \limsup_{n\to\infty} \frac{\|b_{2n+i}\vp_{2n+i}\|r^n}{x_{n}} 
    \leq & \limsup_{n\to\infty} \frac{\|b_{2n+i}\|\|\vp_{2n+i}\|r^n}{x_n} \leq \limsup_{n\to\infty} \frac{(2n+1+i)\|\vp_{2n+i}\|r^n}{x_n} \\
    \leq & \limsup_{n\to\infty} \frac{2^n\|\vp_{2n+i}\|r^n}{x_n} = \limsup_{n\to\infty} \frac{\|\vp_{2n+i}\|(2r)^n}{x_n} = 0.
  \end{align*}
  In other words, $b\vp \in C^{1-i}_{\C{E}(x_n)}$.  The proof works
  for the $B$ operator \emph{mutadis mutandis}.
\end{proof}

We are going to use
  $HP^\ast_{\C{E}(x_n)}(\C{C}^\bullet)$ to denote the cohomology of
  this subcomplex.

\subsection{Asymptotic cohomology of mixed complexes}

There is an analogous asymptotic cohomology for mixed complexes. For a
more detailed treatment on the subject, for the specific case of the
entire cohomology, we refer the reader to~\cite{KlimLesn93}.

Given a mixed complex $(X_*,b_*,B_*)$ one has the differential complexes 
\[ C_\infty^i(X_*) = \prod_{n=0}^\infty X_n \quad\text{ and }\quad
  C^i(X_*) = \bigoplus_{n=0}^\infty X_n \] together with the
differential $b + B$. Similar to the asymptotic differential complex
$C^*_{\C{E}(x_n)}(\C{C}^\bullet)$ associated to a cocyclic module
$\C{C}^\bullet$, for any sequence of non-negative real numbers
$(x_n)_{n\geq 0}$ one can also define the differential subcomplex
\[ C_{\C{E}(x_n)}^i(X_*) := \{ (\vp_{2n+i})_{n\geq 0}\in C_\infty^i(X_*) |\
  \|\vp_{2n+i}\|\prec (x_n)\}. \]  

\subsection{Entire cyclic cohomology as asymptotic cohomology}\label{EntireAsAsymptotic}

As it is indicated in Proposition \ref{prop-subcomplex}, now there is
a whole gamut of asymptotic subcomplexes between
$C^*_{\oplus}(\C{C}^\bullet)$ and $C^*_\infty(\C{C}^\bullet)$ for a
given cocyclic object $\C{C}^\bullet$.  

We recall from \cite{Khalkhali-thesis}, see also \cite{Conn88}, that given a cocyclic object $\C{C}^\bullet$ in $\Banach$, an even (resp. odd) infinite cochain
  $(\vp_{2n})_{n\geq 0} \in C^0_\infty(\C{C}^\bullet)$ (resp. $(\vp_{2n+1})_{n\geq 0} \in C^1_\infty(\C{C}^\bullet)$) is called entire if the radius
  of convergence of the power series
  \begin{equation*} 
    \sum_{n\geq 0}\frac{(2n)!\,\|\vp_{2n}\|}{n!}z^n \qquad \left(\text{\rm resp.}\,\,\sum_{n\geq 0}\frac{(2n+1)!\,\|\vp_{2n+1}\|}{n!}z^n\right)
  \end{equation*} 
  is infinite. Following the common notation of $C_\ve^i(\C{C}^\bullet)$, $i = 0,1$, for the space of (even, and odd) entire cochains, we have 
  \[
  C_\ve^i(\C{C}^\bullet) = C^i_{\C{E}\left(\frac{n!}{(2n)!}\right)}(\C{C}^\bullet)  \subset C^i_{\C{E}(1)}(\C{C}^\bullet).
  \]

In particular, if $\C{C}^\bullet$ is the standard cocyclic object
associated to a unital Banach algebra $A$, then
$C^\ast_\ve(\C{C}^\bullet)$ is
precisely the Connes' entire subcomplex $C^\ast_\ve(A)$, and
$HP^\ast_\ve(\C{C}^\bullet)$ is nothing but the entire
cyclic cohomology $HP_\ve(A)$ of the algebra $A$.

In case $\C{C}^\bullet$ is the cocyclic object associated to a
(Banach-)Hopf algebra $H$, together with a modular pair $(\d,\s)$ in
involution, \cite{ConnMosc98}, then we shall call
$C^\ast_\ve(\C{C}^\bullet)$ the entire Hopf-cyclic complex of $H$, and
$HP^\ast_\ve(\C{C}^\bullet)$ the entire Hopf-cyclic cohomology of
$H$. More generally, we will use the notation
$C^\ast_{\C{E}(x_n)}(H,\d,\s)$ to denote the asymptotic Hopf-cyclic
complex associated with the sequence $(x_n)_{n\geq 0}$, and
$HP^\ast_{\C{E}(x_n)}(H,\d,\s)$ for the cohomology of this complex.

\section{Cohomological functors and asymptotic cohomology}\label{CohomologicalFunctors}

\subsection{Cohomological functors on the category of mixed complexes}

A $\B{Z}$-graded functor of the form $F_*\colon\Mix\to \vect$ is
called a \emph{cohomological $\delta$-functor}~\cite{Weib-book} if
every short exact sequence of mixed complexes
\[ 0\to (X_*,b_*,B_*)\to (X'_*,b'_*,B'_*)\to (X''_*,b''_*,B''_*)\to 0 \] 
is sent to a long exact sequence of $\B{C}$-modules
\[ \cdots \to F_n(X_*,b_*,B_*)\to F_n(X'_*,b'_*,B'_*)\to
F_n(X''_*,b''_*,B''_*)\to F_{n+1}(X_*,b_*,B_*)\to \cdots \]

Given a cohomological $\delta$-functor $F_*\colon \Mix\to\vect$, a
mixed complex $(X_*,b_*,B_*)$ is called \emph{$F$-acyclic} if
$F_n(X_*,b_*,B_*)=0$ for every $n\in\B{Z}$.  Similarly, we are going
to call a morphism of mixed complexes
$f_*\colon (X_*,b_*,B_*)\to (X'_*,b'_*,B'_*)$ a \emph{$F$-equivalence}
if the induced morphisms of $\B{C}$-modules
$F_n(f_*)\colon F_n(X_*,b_*,B_*)\to F_n(X'_*,b'_*,B'_*)$ are all
isomorphisms.

\begin{proposition}\label{Contraction}
Let $(X_*,b_*,B_*)$ be a mixed complex such that $(X_*,b_*)$ is
  acyclic.  Then $(X_*,b_*,B_*)$ is $F$-acyclic for any cohomological
  $\delta$-functor $F_*\colon \Mix\to \vect$.
\end{proposition}

\begin{proof}
  There is a short exact sequence of mixed complexes of the form
  \[ 0\to (\ker(b_*),0,B_*)\to (X_*,b_*,B_*)\xrightarrow{\pm b_*}
  (\Im(b_*),0,B_*) \to 0 \]
  Then for a cohomological $\delta$-functor
  $F_*\colon \Mix\to \vect$ we get a long exact sequence of the form
  \begin{align*}
    \cdots\to F_{n-1}(\Im(b_*),0,B_*)\to & F_n(\ker(b_*),0,B_*)\to F_n(X_*,b_*,B_*)\to \\
    & \to F_n(\Im(b_*),0,B_*) \to F_{n+1}(\ker(b_*),0,B_*)\to \cdots
  \end{align*}
  By acyclicity of $(X_*,b_*)$, the natural embedding
  $\Im(b_*)\to \ker(b_*)[+1]$ is bijective. Thus we get
  $F_n(X_*,b_*,B_*)= 0$ for every $n\in\B{Z}$.
\end{proof}

\begin{proposition}\label{Equivalence}
  Let $f_*\colon (X_*,b_*,B_*)\to(X'_*,b'_*,B'_*)$ be a morphism of
  mixed complexes such that $f_*\colon (X_*,b_*)\to (X'_*,b'_*)$ is a
  quasi-isomorphism.  Then $f_*$ is an $F$-equivalence for every
  cohomological $\delta$-functor $F\colon \Mix\to \vect$.
\end{proposition}

\begin{proof}
  The abelian $\B{C}$-linear category of mixed complexes $\Mix$ is
  isomorphic to the abelian $\B{C}$-linear category of differential
  graded modules over the quotient polynomial $\B{C}$-algebra
  $\C{B}:= \B{C}[B]/(B^2)$.  Given a morphism $f_*\colon X_*\to X'_*$
  of differential graded $\C{B}$-modules, the fact that $f_*$ is a
  quasi-isomorphism in the $b_*$-direction is equivalent to the fact
  that $f_*$ is an ordinary quasi-isomorphism of differential graded
  $\C{B}$-modules.  Now, we form the mapping cone $C(f_*)$ and
  consider
  \[ 0\to (X_*,b_*)\to C(f_*) \to K(f_*) \to 0\]
  where $K(f_*)$ is the cokernel of the embedding
  $(X_*,b_*)\to C(f_*)$.  Since $f_*$ is a quasi-isomorphism, $K(f_*)$
  is acyclic.  Then the result follows from
  Proposition~\ref{Contraction}.
\end{proof}

\subsection{Bounded mixed complexes}

For an index $n\in\B{Z}$ and a mixed complex $(X_*,b_*,B_*)$, the
\emph{good truncations} $(X_{*\leq n},b_*,B_*)$ and
$(X_{*>n},b_*,B_*)$ of $X_*$ are defined to be
\[ X_{m\leq n} =
\begin{cases}
  X_m & \text{ if } m\leq n\\
  \Im(b_m) & \text{ if } m=n+1\\
  0       & \text{ otherwise}
\end{cases}
\quad\text{ and }\quad
X_{m>n} =
\begin{cases}
  0 & \text{ if }m\leq n\\
  X_{m+1}/\Im(b_m) & \text{ if } m = n+1\\
  X_m & \text{ otherwise}
\end{cases}\]
These mixed complexes fit into a short exact sequence of the form
\begin{equation}\label{TruncationSequence}
  0 \to (X_{*\leq n},b_*,B_*) \to (X_*,b_*,B_*) \to (X_{*>n},b_*,B_*) \to 0
\end{equation}
For the homology in the $b_*$ direction, we get
\[ H_m(X_{*\leq n},b_*) =
\begin{cases}
  H_m(X_*,b_*) & \text{ if } m\leq n,\\
  0           & \text{ otherwise},
\end{cases}
\quad\text{ and }\quad
H_m(X_{*>n},b_*) =
\begin{cases}
  H_m(X_*,b_*) & \text{ if } m>n\\
  0           & \text{ otherwise}
\end{cases}
\]

A mixed complex $(X_*,b_*,B_*)$ is called \emph{bounded} if there is
an index $N$ such that the good truncation $(X_{*>N},b_*)$ is acyclic.
Similarly, a (co)cyclic object $\C{C}^{\bullet}$ is called bounded if
its associated mixed complex $\C{M}(\C{C}^{\bullet})$ is bounded.

\begin{proposition}\label{Collapse}
  If $HP^*_{\C{E}(x_n)}$ is a cohomological $\delta$-functor on the
  category of mixed complexes, then $HP^*_{\C{E}(x_n)}$ must be the
  same as the algebraic periodic cohomology $HP^*$ on the subcategory
  of bounded mixed complexes.
\end{proposition}

\begin{proof}
  Assume we have a bounded mixed complex $(X_*,b_*,B_*)$.  Then for a
  large enough $n$, the natural embedding of mixed complexes
  $(X_{*\leq n},b_*,B_*)\to (X_*,b_*,B_*)$ induces a homotopy
  equivalence of mixed complexes since
  $(X_{*\leq n},b_*)\to (X_*,b_*)$ is a quasi-isomorphism.  On the
  other hand, in view of the assumption,
  replacing $(X_*,b_*,B_*)$ with its good truncation
  $(X_{*\leq n},b_*,B_*)$ for a large enough $n$, we obtain a
  homotopical equivalence of the form
  \[ C^*_{\C{E}(x_n)}(X_{*\leq n},b_*,B_*)\xra{\simeq}
    C^*_{\C{E}(x_n)}(X_*,b_*,B_*). \] However, the bicomplex computing
  $C^*_{\C{E}(x_n)}(X_{*\leq n},b_*,B_*)$ is confined within a bounded
  strip along the $p=q$ line.  This means
  $C^*_{\C{E}(x_n)}(X_{*\leq n},b_*,B_*)$ is \emph{equal} to the
  algebraic periodic complex $C^*(X_{*\leq n},b_*,B_*)$ for large
  enough $n$.  On the other hand, $C^*(X_{*\leq n},b_*,B_*)$ is
  homotopy equivalent to $C^*(X_*,b_*,B_*)$ since the algebraic
  periodic cyclic cohomology $HP^{*}$ is a cohomological
  $\delta$-functor on bounded mixed complexes.  This follows from the
  fact that the ordinary cyclic cohomology $HC^{*}$ of cocyclic
  modules is a cohomological
  $\delta$-functor~\cite[Prop. 1.3]{Kassel:1987}, and that the periodic
  cyclic cohomology is the stabilization of $HC^{*}$ under the 
  periodicity operator $S$.  Then for large enough $m$,
  $HC^{m}(X_{*})$ and $HP^{m \,({\rm mod}\, 2)}(X_{*})$ are the same for bounded
  mixed complexes.  The result follows.
\end{proof}

\subsection{The cyclic bicomplex and the stability
  phenomenon}\label{Tsygan}

Instead of using Connes' $(b,B)$-complex, one can use the
cyclic bicomplex in studying cyclic cohomology.  In this subsection,
we will follow this route.  

Let $\C{C}^\bullet$ be a asymptotically normalized cocyclic object in $\Banach$,
and consider the cyclic bicomplex
\[ C^{p,q} = C^q \qquad 
   d^{p,q}_v = 
   \begin{cases}
     b_q & \text{ if $p$ is even },\\
     b'_q & \text{ if $p$ is odd },
   \end{cases}  \qquad
   d^{p,q}_h = 
   \begin{cases}
    (1-t_q) & \text{ if $p$ is even },\\
    N_q & \text{ if $p$ is odd }.
   \end{cases}
\]
One can similarly define the product total complex
\[ {\rm Tot}_\Pi^n(\C{C}^\bullet) = \prod_{m=0}^\infty C^m \]
together with the differential
$d_n\colon {\rm Tot}^n_\Pi\to {\rm Tot}^{n+1}_\Pi$ coming from
$d^{*,*}_v$ and $d^{*,*}_h$.  We now define asymptotic subcomplexes
by imposing a growth condition:
\[ {\rm Tot}_{\C{E}(x_n)}^n(\C{C}^\bullet) = \left\{(\vp_n)_{n\geq 0}\in
  \prod_{m=0}^\infty C^m:\ (\|\vp_n\|)_{n\geq 0}\in\C{E}(x_n) \right\}. \]
The cohomology of this new complex yields a different asymptotic
cohomology theory for cocyclic objects.  We shall reserve the notation $HS^{*}_{\C{E}(x_{n})}$ to this new
cohomology theory.

Let us consider the following definition.
\begin{definition}
  A morphism of cocyclic modules
  $f^\bullet\colon \C{C}^\bullet\to \C{D}^\bullet$ is called \emph{a
    stable isomorphism} if there is an index $N\geq 0$ such that the
  induced morphism
  $H^n(f^\bullet)\colon H^n(\C{C}^\bullet)\to H^n(\C{D}^\bullet)$ on
  the Hochschild cohomology is an isomorphism for every $n\geq N$.
\end{definition}

\begin{proposition}\label{StableIsomorphism}
Let $f^\bullet\colon \C{C}^\bullet\to \C{D}^\bullet$ be a
  stable isomorphism of cocyclic objects. If
  $HS_{\C{E}(x_n)}^*$ is a cohomological functor, then $f^\bullet$
  induces an isomorphism in cohomology
  \[ HS_{\C{E}(x_n)}^*(f^\bullet)\colon
  HS_{\C{E}(x_n)}^*(\C{C}^\bullet)\to
  HS_{\C{E}(x_n)}^*(\C{D}^\bullet). \]
\end{proposition}

\begin{proof}
  Now, for every $p\geq 0$ let
  \[ L_{\geq p}^*(\C{C}^\bullet,\C{E}(x_n)) = \left\{(\vp_n)_{n\geq p}\in
    \prod_{m=p}^\infty C^m:\ (\|\vp_n\|)_{n \geq p}\in\C{E}(x_n) \right\}. \]
  Then we have a short exact sequence of complexes of the form
  \[ 0 \to L_{\geq p}^*(\C{C}^\bullet,\C{E}(x_n)) \to
  {\rm Tot}^*_{\C{E}(x_n)}(\C{C}^\bullet)\to L_{<p}^*(\C{C}^\bullet) \to
  0 \] where we have
  \[ L_{<p}^*(\C{C}^\bullet) = \left\{(\vp_n)_{n\geq 0}\in \prod_{m=0}^\infty
    C^m:\ \vp_n=0, \text{ for every } n\geq p \right\}. \]
  Notice that the subcomplex $L_{<p}^*(\C{C}^\bullet)$ is the same for
  every sequence $(x_n)_{n\geq 0}$, and therefore, is independent of the given
  sequence.  Since $L_{<p}^*(\C{C}^\bullet)$ is a bounded double
  complex whose rows are exact, it is acyclic.  In other words, there is a
  natural quasi-isomorphism of the form
  \[ \lim_{\longleftarrow p} L_{\geq p}^*(\C{C}^\bullet,\C{E}(x_n))
  \to {\rm Tot}^*_{\C{E}(x_n)}(\C{C}^\bullet). \]
  The result follows.
\end{proof}  

\begin{proposition}\label{TrivialityOnBounded}
  If $HS^{*}_{\C{E}(x_{n})}$ is a cohomological $\delta$-functor then
  $HS^{*}_{\C{E}(x_{n})}$ must be trivial on the subcategory of
  bounded cocyclic objects.
\end{proposition}

\begin{proof}
  By definition, bounded objects are stably isomorphic to the ground
  field. 
\end{proof}

\begin{remark}
  We observe that the well-behaved asymptotic analogues of the $(b,B)$ and
the cyclic bicomplexes diverge substantially: the former collapses onto
  the algebraic periodic cohomology as shown in
  Proposition~\ref{Collapse} while the latter becomes trivial as shown
  in Proposition~\ref{StableIsomorphism}, on the subcategory of
  bounded cocyclic objects.
\end{remark}

\begin{remark}\label{StableCohomology}
  Recall that in passing from the ordinary cyclic cohomology to the
  periodic cyclic cohomology, we replace the group cohomologies of
  cyclic groups with the Tate cohomology of the cyclic groups along
  the rows.  The net effect is that we kill zero-divisors in the group
  cohomology of cyclic groups, or equivalently, we use a cohomology
  theory which is stable in that direction using the periodicity
  operator $S$.  Now, as we show in
  Proposition~\ref{StableIsomorphism}, we need a cohomology theory
  which is stable also in the Hochschild direction.  Moreover, we also
  see that \emph{all} such well-behaved asymptotic cohomologies of
  cocyclic objects are determined up to stable isomorphisms, as
  opposed to ordinary quasi-isomorphisms of cocyclic objects. 
\end{remark}

\section{Products in asymptotic cyclic cohomology}\label{Products}

\subsection{Products of (co)cyclic modules, and mixed complexes}

The category of cocyclic objects is a strict monoidal category with
the monoidal product of two cocyclic modules defined diagonally.
Namely, if $\C{C}^\bullet$ and $\C{D}^\bullet$ are two cocyclic
modules, their product $\diag(\C{C}^{\bullet}\cotimes\C{D}^{\bullet})$
is defined to be the graded module
$\bigoplus_{n\in\B{N}} \C{C}^n\cotimes\C{D}^n$ together with
\begin{align*}
  d_i^\C{C}\cotimes d_i^\C{D}\colon&\C{C}^n\cotimes\C{D}^n \lra \C{C}^{n+1}\cotimes\C{D}^{n+1},  
                                     \quad 0\leq i \leq n+1, \\
  s_j^\C{C}\cotimes s_j^\C{D}\colon&\C{C}^n\cotimes\C{D}^n \lra \C{C}^{n-1}\cotimes\C{D}^{n-1},  
                                     \quad 0\leq j \leq n-1, \\
  t_n^\C{C}\cotimes t_n^\C{D}\colon&\C{C}^n\cotimes\C{D}^n \lra \C{C}^n\cotimes\C{D}^n. 
\end{align*}

Similarly, the category of mixed complexes has their own strict monoidal
product defined as follows: Given two mixed complexes
$(\C{C}^*,b^\C{C}_*,B^\C{C}_*)$ and $(\C{D}^*,b^\C{D}_*,B^\C{D}_*)$,
the product complex is the graded product of these mixed complexes
\[ (\C{C}\cotimes\C{D})^n = \bigoplus_{p+q=n} \C{C}^p\cotimes\C{D}^q \]
together with the differentials
\[ b_n = \sum_{p+q=n} (b^\C{C}_p\cotimes id_q) + (-1)^p(id_p\cotimes b^\C{D}_q) \quad\text{ and }\quad
   B_n = \sum_{p+q=n} (B^\C{C}_p\cotimes id_q) + (-1)^p(id_p\cotimes B^\C{D}_q).
\]
We shall denote the first summands by $\hb$ and $\hB$, and the second
summands by $\vb$ and $\vB$, respectively.

\subsection{Cup product in asymptotic cyclic cohomology}

The functor that sends a cocyclic object $\C{C}^\bullet$ to its mixed
complex $\C{M}(\C{C}^\bullet)$ is weakly monoidal.  In other words,
for every pair of cocyclic objects $\C{C}^{\bullet}$ and
$\C{D}^{\bullet}$, there are natural quasi isomorphisms of the form
\[ Sh_{\C{C}^\bullet,\C{D}^\bullet}\colon
\C{M}(\diag(\C{C}^\bullet\cotimes\C{D}^\bullet)) \xra{}
\C{M}(\C{C}^\bullet)\cotimes\C{M}(\C{D}^\bullet) \]
implemented by the cyclic shuffle maps.  This follows from the cyclic
Eilenberg-Zilber Theorem. See \cite[Thm. 4.3.8]{Loday-book},
\cite{GetzJone93}, \cite{KhalRang04}, or \cite{KRT:HopfCyclicCohomologyOfQuantumGroups}.

Next, we show that the functor which sends a cocyclic Banach module to
an asymptotic complex is weakly monoidal in the following sense.
\begin{proposition}\label{Diagonal}
  Assume $\C{C}^{\bullet}$ and $\C{D}^{\bullet}$ are two cocyclic
  object, and let $(x_{n})_{n\geq 0}$ and $(y_{n})_{n\geq 0}$ be two non-decreasing
  sequences of positive real numbers.  Then there is a natural
  morphism of differential graded $\B{C}$-vector spaces of the form
  \begin{equation}
      C^*_{\C{E}(x_{n})}(\C{C}^{\bullet})\cotimes C^*_{\C{E}(y_{n})}(\C{D}^{\bullet})
      \xra{}C^*_{\C{E}(x_{n}y_{n})}(\diag (\C{C}^{\bullet}\cotimes \C{D}^{\bullet})) 
  \end{equation}
\end{proposition}
\begin{proof}
  The given map is the usual external cup product in cyclic
  cohomology.  Thus its compatibility with the differential maps is
  immediate.  So, we only need to show that it does land in the right
  subcomplex in the target. To this end we observe that
  \begin{equation*}
    {\bf u}_{k}\cup{\bf v}_{n-k}
    = \sum_{\mu \in Sh(k,n-k)} d_{\wbar{\mu}(n)}\ldots d_{\wbar{\mu}(k+1)}{\bf u}_{k}\otimes d_{\wbar{\mu}(k)}\ldots d_{\wbar{\mu}(1)}{\bf v}_{n-k}
  \end{equation*}
  and that
  \begin{equation*}
    \left\|\sum_{k=0}^n {\bf u}_{k}\cup{\bf v}_{n-k}\right\| 
    \leq \sum_{k=0}^n \|{\bf u}_{k}\cup{\bf v}_{n-k}\| 
    \leq \sum_{k=0}^n \binom{n}{k} \|{\bf u}_{k}\|\|{\bf v}_{n-k}\|. 
  \end{equation*}
  The second inequality follows from the normalization of the cyclic
  objects. As a result, we get
  \begin{equation*}
    \frac{\left\|\sum_{k=0}^m {\bf u}_{k}\cup{\bf v}_{m-k}\right\| r^m}{x_my_m} 
    \leq \sum_{k=0}^m\binom{m}{k}\frac{\|{\bf u}_{k}\|\|{\bf v}_{m-k}\| r^m}{x_my_m} 
    \leq \sum_{k=0}^m\frac{\|{\bf u}_{k}\|\|{\bf v}_{m-k}\|(2r)^m}{x_my_m}.
  \end{equation*}
  On the other hand, since the sequences $(x_n)_{n\geq 0}$ and $(y_n)_{n\geq 0}$ are non-decreasing,
  \begin{equation*}
    x_k y_{n-k} \leq x_ny_n, \qquad 0\leq k\leq n,
  \end{equation*} 
thus we get
  \begin{equation*}
    \frac{\left\|\sum_{k=0}^m {\bf u}_{k}\cup{\bf v}_{m-k}\right\| r^m}{x_my_m} 
    \leq \sum_{k=0}^m \frac{\|{\bf u}_{k}\|\|{\bf v}_{m-k}\|(2r)^m}{x_ky_{m-k}}.
  \end{equation*}
In total, we have
  \begin{equation*}
    \sup_{m\geq n} \frac{\left\|\sum_{k=0}^m {\bf u}_{k}\cup{\bf v}_{m-k}\right\| r^m}{x_my_m} 
    \leq \sup_{m\geq n}\sum_{k=0}^m \frac{\|{\bf u}_{k}\|(\sqrt{2}r)^{k}}{x_{k}} 
    \cdot \sup_{m\geq n}\sum_{k=0}^m \frac{\|{\bf v}_{k}\|(\sqrt{2}r)^k}{y_{k}},
  \end{equation*}
from which the claim follows by applying the limit $n\to \infty$.  
\end{proof}

\subsection{Cup product in asymptotic Hopf-cyclic cohomology} 

Let $\C{B}$ be a (Banach-)Hopf algebra with a fixed MPI $(\delta,\sigma)$, and let $C$ be a
Banach $\C{B}$-module coalgebra.  In other words, 
\begin{equation*}
\D(h(c)) = h\ps{1}(c\ps{1})h\ps{2}(c\ps{2}), \qquad \ve(h(c)) = \ve(h)\ve(c), \qquad \|h(c)\| \leq \|h\|\|c\|,
\end{equation*}
for every $h\in\C{B}$ and $c\in C$.  Now, assume $A$ is a Banach
$\C{B}$-module algebra which means that we have
\begin{equation}\label{module-alg}
  h(ab) = h\ps{1}(a)h\ps{2}(b), \qquad h(1) = \ve(h)1, \qquad \|h(a)\| \leq \|h\|\|a\|,
\end{equation}
for any $h\in \C{B}$, and any $a,b \in A$.  We are going to say that
\emph{$A$ admits an $\C{B}$-equivariant action of $C$} if there is a
map $C \ot A \to A$ satisfying
\begin{equation}\label{coalg-alg-pairing}
  c(ab) = c\ps{1}(a)c\ps{2}(b),\qquad (S^{-1}(h)(c))(a) = c(h(a)), \qquad c(1) = \ve(c)1, \qquad \|c(a)\| \leq \|c\|\|a\|
\end{equation}
for any $c\in C$, and any $a,b \in A$. Now, let
\begin{enumerate}[(i)]
\item $\C{C}^{\bullet}_{\C{B}}(A;\delta,\sigma)$ be the standard Hopf-cocyclic
  object associated with the $\C{B}$-module algebra $A$ with
  coefficients in the MPI $(\delta,\sigma)$,
\item $\C{C}^{\bullet}_{\C{B}}(C;\delta,\sigma)$ be the standard Hopf-cocyclic
  object associated with the $\C{B}$-module coalgebra $C$ with
  coefficients in the MPI $(\delta,\sigma)$,
\item and $\C{C}^{\bullet}(A)$ be the standard cocyclic object associated
  with the algebra $A$.
\end{enumerate}

The proof of the following Lemma is a straightforward but tedious
check of the compatibility conditions between the cocyclic structure
maps, and therefore, is omitted.
\begin{lemma}\label{DiagonalAction}
  If $A$ admits a $\C{B}$-equivariant action of $C$, then there is a
  well-defined morphism of cocyclic objects of the form
  \[ \Gamma^{\bullet}\colon\diag\left(\C{C}^{\bullet}_{\C{B}}(C;\delta,\sigma)\cotimes\,
     \C{C}^{\bullet}_{\C{B}}(A;\delta,\sigma)\right) \to \C{C}^{\bullet}(A)\]
  where the range is the standard cocyclic object associated with an
  algebra $A$ where we define
  \begin{equation}
    \Gamma^{n}(c^{0}\otimes\cdots\otimes c^{n}\mid \varphi)(a_{0}\otimes\cdots\otimes a_{n})
    = \varphi(c^{0}(a_{0}),\ldots,c^{n}(a_{n}))
  \end{equation}
  for every $c^{0}\otimes\cdots\otimes c^{n}\in\C{C}^{n}_{\C{B}}(C;\delta,\sigma)$ and $\varphi\in\C{C}^n_{\C{B}}(A;\delta,\sigma)$.
\end{lemma}

Using Lemma~\ref{DiagonalAction}, in combination with
Proposition~\ref{Diagonal}, we get the following.
\begin{theorem}\label{char-hom}
  Let $\C{B}$ be a (Banach-)Hopf algebra with an MPI $(\delta,\sigma)$, $A$ a
  unital Banach $\C{B}$-module algebra, and $C$ a Banach
  $\C{B}$-module coalgebra such that $A$ admits a $\C{B}$-equivariant
  action of $C$.  Assume also that $(x_n)_{n\geq 0}$ and
  $(y_n)_{n\geq 0}$ are two non-decreasing sequences in $(0,\infty)$.
  Then there is a cup product of the form
  \[ \cup\colon HP^{i}_{\C{B},\C{E}(x_{n})}(C)\otimes
  HP^{j}_{\C{B},\C{E}(y_{n})}(A)\to HP^{i+j}_{\C{E}(x_{n}y_{n})}(A)\]
  where $HP^{i}_{\C{B},\C{E}(x_{n})}(C)$ and
  $HP^{i}_{\C{B},\C{E}(x_{n})}(A)$ respectively denote the asymptotic
  Hopf-cyclic cohomologies of $C$ and $A$, while
  $HP^{i}_{\C{E}(x_{n})}(A)$ denotes the asymptotic cyclic cohomology
  of $A$.
\end{theorem}


\section{Asymptotic cyclic cohomology of the simplex}\label{AsymptoticSimplex}

Let $\Delta^n$ denote the geometric $n$-simplex, and let
\begin{equation}\label{simplex} 
\Delta^\bullet := \bigoplus_{n\geq 0} \B{C}[\Delta^n].
\end{equation}
In this subsection we are going to introduce a non-trivial cocycle in
the $\C{E}(1)$-asymptotic cyclic cohomology of the cocyclic module
\eqref{simplex}.

We recall from \cite{GetzJone93} that the geometric $n$-simplex can be described in two different coordinate systems
  \begin{equation}
    (t_0,\ldots,t_{n+1}) \quad\text{ where }\quad 0=t_0\leq t_1\leq \cdots \leq t_n \leq t_{n+1} = 1,
  \end{equation}
  and
  \begin{equation}
    (t_0,\ldots,t_n) \quad\text{ where $t_i\in [0,1]$ and } 1 = \sum_{i=0}^n t_i.
  \end{equation}

Notice that the 0-simplex contains one single point.  In the first
coordinate system this is represented by the sequence $(0,1)$ while in
the second it is simply $1$.  We will denote this point by $*$
independent of the coordinate system chosen.  We will prefer the first
coordinate system for our calculations below.  Also, in writing an
element $(t_{0},\ldots,t_{n+1})$ we will drop $t_{0}=0$ and
$t_{n+1}=1$ from the coordinates for convenience.

We leave the proof of the following fact to the reader:
\begin{lemma}\label{SimplexCocyclicStructure}
The graded space \eqref{simplex} has a cocyclic structure determined by the coface operators
$\delta_i^n:\Delta^n \lra \Delta^{n+1}$ 
\begin{align*}
 \delta_i(t_1, \ldots, t_n) = 
  \begin{cases}
    (0, t_1, \ldots, t_n) & \text{ if } i=0,\\
    (t_1, \ldots, t_i, t_i, \ldots, t_n) & \text{ if } 0 \leq i \leq n, \\
    (t_1, \ldots, t_n, 1) & \text{ if } i=n+1,
  \end{cases}
\end{align*}
the codegeneracy operators $\s_j^n\colon\delta^n \lra \delta^{n-1}$
\begin{align*}
& \s_j^n(t_1, \ldots, t_n) = (t_1, \ldots, \widehat{t_{j+1}}, \ldots, t_n),
\end{align*}
and the cyclic operators $\tau_n:\Delta^n \lra \Delta^n$
\begin{equation*}
\tau_n(t_1, \ldots, t_n) = (t_2-t_1, t_3-t_1, \ldots, t_n-t_1, 1-t_1).
\end{equation*}
defined for $0 \leq i \leq n+1$ and $0 \leq j \leq n-1$.
\end{lemma}

Given any sequence of positive numbers $(x_{n})_{n\geq 0}$, in the construction of the asymptotic complex
\begin{align*}
\xymatrix{
\cdots \ar[r] & C_{\C{E}(x_{n})}^0(\Delta^{\bullet}) \ar[r]^{b+B} 
& C_{\C{E}(x_{n})}^1(\Delta^{\bullet}) \ar[r]^{b+B} 
& C_{\C{E}(x_{n})}^0(\Delta^{\bullet}) \ar[r] & \cdots
}
\end{align*}
we endow the $n$-simplex $\Delta^{n}$ with the norm $\|(t_1,\ldots,t_n)\| := \max\{t_i|\ i=1,\ldots,n\}$.

\begin{proposition}\label{UniversalIndexCocycle}
  The cochain $\vp=(\vp_{2n})_{n\geq 0} \in C^0_\infty(\Delta^{\bullet})$ of infinite support, given by 
  \begin{equation}\label{cocycle-D}
    \vp_{2n} =
    \begin{cases}
      \ast, & \text{\rm if } n=0,\\
      \frac{(-1)^n}{2^nn!} \sum_{r=0}^n \tau_{2n}^{2r}\delta_{0}^{2n}(*) & \text{\rm if } n\geq 1,
    \end{cases}
  \end{equation}
  is a non-trivial $\C{E}(1)$-asymptotic cocycle.
\end{proposition}

\begin{proof}
Let us first note that
\[
\sqrt[n]{\|\vp_{2n}\|} \leq \sqrt[n]{\frac{n+1}{2^nn!}} = \frac{\sqrt[n]{n+1}}{2\sqrt[n]{n!}}.
\]
Then from Stirling's approximation we get
\[
\limsup_{n\to \infty}\,\|\vp_{2n}\| = 0,
\]
which implies $(\vp_{2n})_{n\geq 0}\in C^0_{\C{E}(1)}(\Delta^{\bullet})$.

For $n > 1$, we see that
  \begin{align*}
    B(\delta_{0}^{2n}(*)) 
    = & N_{2n-1}\s_{2n-1}\tau_{2n}(\Id - \tau_{2n})\delta_{0}^{2n}(*) \\
    = & N_{2n-1}\s_{2n-1}((\underset{2n-1-many}{\underbrace{0,\ldots,0}}, 1) - (\underset{2n-2-many}{\underbrace{0,\ldots,0}}, 1, 1))\\
    = & (\Id - \tau_{2n-1} + \tau_{2n-1}^2 - \ldots - \tau_{2n-1}^{2n-1})((\underset{2n-1-many}{\underbrace{0,\ldots,0}}) - (\underset{2n-2-many}{\underbrace{0,\ldots,0}}, 1)) \\
    = & (\Id - \tau_{2n-1} + \tau_{2n-1}^2 - \ldots - \tau_{2n-1}^{2n-1})(\Id - \tau_{2n-1})\delta_{0}^{2n-1}(*)\\
    = & 2 N_{2n-1}\delta_{0}^{2n-1}(*).
  \end{align*}
  Similarly, we get
  \begin{align*}
    B(\tau_{2n}^{2r}\delta_{0}^{2n}(*)) 
    = & N_{2n-1}\s_{2n-1}\tau_{2n}(\Id - \tau_{2n})\tau_{2n}^{2r}\delta^{2n}_{0}(*)\\
    = & N_{2n-1}(\tau_{2n-1}^{2r} - \tau_{2n-1}^{2r+1})\delta_{0}^{2n-1}(*) \\
    = & 2 N_{2n-1}\delta_{0}^{2n-1}(*),
  \end{align*}
  and
  \begin{align*}
    B(\tau_{2n}^{2n}\delta_{0}^{2n}(*))
    = & N_{2n-1}\s_{2n-1}\tau_{2n}(\Id - \tau_{2n})\tau_{2n}^{2n}\delta_{0}^{2n}(*)\\
    = & N_{2n-1}\s_{2n-1}(\Id-\tau_{2n})\delta_{0}^{2n}(*) \\
        = & N_{2n-1}(\Id-\tau_{2n})\delta_{0}^{2n}(*) = 0 .
  \end{align*}
  On the other hand, we have
  \begin{align*}
    b(\tau_{2n}^{2r}\delta_{0}^{2n}(*))
    = & \tau_{2n}^{2r}\delta_{0}^{2n}(*) - \tau_{2n}^{2r-1}\delta_{0}^{2n}(*).
  \end{align*}
  As a result, we see the following picture:
  \[
  \xymatrix{
    & \vdots & & & & & & \\
    & \frac{1}{2^nn!} \sum_{r=0}^n\tau_{2n}^{2r}\delta_{0}^{2n}(*) \ar[u]^b \ar[r]^{B} &  N_{2n-1}\delta_{0}^{2n-1}(*) &   &  &  & \\
    & & \frac{1}{2^{n-1}(n-1)!} \sum_{r=0}^{n-1}\tau_{2n-2}^{2r}\delta_{0}^{2n-2}(*) \ar[u]^b \ar[r]^{\hspace{2cm}B}  & \cdots   &   &  & \\
  }
  \]
  As for the non-triviality, we observe that for any $0 \leq s \leq 2n+1$,
  \begin{align*}
    B_{2n+1}\tau_{2n}^s\delta_{0}^{2n+1}(*)
    = & N_{2n}\s_{2n}\tau_{2n} (\Id + \tau_{2n})\ \tau_{2n}^s\delta_{0}^{2n+1}(*) \\
    = & N_{2n}(\tau_{2n}^s + \tau_{2n}^{s+1}) \delta_{0}^{2n+1}(*)\\
    = & 2N_{2n}\delta_{0}^{2n+1}(*),
  \end{align*}
  whereas
  \begin{align*}
    b_{2n-1}\tau_{2n}^s\delta_{0}^{2n-1}(*)
    = & \tau_{2n}^s\delta_{0}^{2n-1}(*) = \begin{cases}
      \tau_{2n}^{s+1}\delta_{0}^{2n}(*), & \text{\rm if}\,\, s \, \text{\rm is even}, \\
      \tau_{2n}^s \delta_{0}^{2n}(*), & \text{\rm if}\,\, s \, \text{\rm is odd}.
    \end{cases}
  \end{align*}
  The claim then follows.
\end{proof}

\begin{remark}
Note that the cocycle $(\vp_{2n})_{n\geq 0}\in C^0_{\C{E}(1)}(\Delta^{\bullet})$ is not entire, that is, $(\vp_{2n})_{n\geq 0} \notin C^0_\ve(\Delta^{\bullet})$. Indeed,
\[
\sum_{n\geq 0}\frac{(2n)!\|\vp_{2n}\|}{n!}|z|^n \geq \sum_{n\geq 0}\frac{(2n)!}{2^nn!n!}|z|^n,
\]
where the latter has radius of convergence $\frac{1}{2}$.
\end{remark}

\section{Asymptotic characteristic map and the index cocycles}\label{ChernConnesCharacter}

\subsection{Theta-summable Fredholm modules and the Chern character} 

We are going to recall the construction of the Chern character formula of a theta-summable Fredholm module from \cite{JaffLesnOste88,GetzSzen89}, see also \cite{Conn88}.

\begin{definition}\label{def-Fredholm}
A theta-summable Fredholm module over a unital Banach algebra $\C{A}$ is a pair $(\C{H},\slashed{D})$ consisting of a $\B{Z}_2$-graded Hilbert space $\C{H} = \C{H}^+\oplus \C{H}^-$, admitting a continuous representation of $\C{A}$, and an odd self-adjoint operator $\slashed{D}:\C{H}^\pm \to \C{H}^\mp$, such that 
\begin{enumerate}[(i)]
\item for any $a\in \C{A}$, the operator $[\slashed{D},a]$ is densely defined, extends to a bounded operator on $\C{H}$, and there is $N(\slashed{D}) >0$ with $\|a\| + \|[\slashed{D},a]\| \leq N(\slashed{D})\|a\|$,
\item ${\rm Tr}\,e^{-(1-\ve)\slashed{D}^2}$ is finite for some $\ve>0$.
\end{enumerate}
\end{definition}

In order to define the Chern character of a Fredholm module, let $e(t):= e^{-t\slashed{D}^2}$, and
\begin{equation}\label{integral}
\langle a_0, \ldots, a_n\rangle_n := \int_{\Delta^n}\,{\rm Str}\left(a_0e(t_1)a_1e(t_2-t_1)a_2\ldots e(t_n-t_{n-1})a_ne(1-t_n)\right)dt_1\ldots dt_n,
\end{equation}
where $\Delta^n$ denote the $n$-simplex given by $(t_1,\ldots, t_n)$ with $0\leq t_1\leq \ldots \leq t_n\leq 1$, and ${\rm Str}(a) := {\rm Tr}\left(a|_{\C{H}^+}\right) - {\rm Tr}\left(a|_{\C{H}^-}\right)$ is the supertrace of an operator $a:\C{H}\to \C{H}$.

The even entire cochain ${\rm Ch}(\slashed{D}) \in C_\ve^0(\C{A})$ given by
\begin{equation*}
{\rm Ch}^{2n}(\slashed{D}) (a_0,\ldots, a_{2n}) := \langle a_0, [\slashed{D},a_1], \ldots, [\slashed{D},a_{2n}] \rangle_{2n}, \qquad n\geq 0
\end{equation*} 
is called the Chern character of a theta-summable Fredholm module $(\C{H},\slashed{D})$ over a unital Banach algebra $\C{A}$. It is proved in \cite{JaffLesnOste88} that $[({\rm Ch}^{2n}(\slashed{D}))_{n\geq 0}] \in HP_\ve^0(\C{A})$ represents a nontrivial even entire class. This cocycle is usually referred as \emph{the JLO-cocycle} in the literature.  

Following \cite{Getz93}, we recall also the odd theta-summable Fredholm module. 
\begin{definition}
Given a Banach $\ast$-algebra $\C{A}$, an odd theta-summable Fredholm module is also a pair $(\C{H},\slashed{D})$ consisting of a Hilbert space $\C{H}$ as a continuous $\ast$-representation of $\C{A}$, and a self-adjoint operator $\slashed{D}:\C{H}\to \C{H}$ so that
\begin{enumerate}[(i)]
\item there is $c >0$ such that $\|[\slashed{D},a]\| \leq c\|a\|$, for all $a \in \C{A}$,
\item if $t>0$, then ${\rm Tr}\,e^{-t\slashed{D}^2}$ is finite.
\end{enumerate}
\end{definition}

Given two self-adjoint operators $A_0$ and $A_1$ on $\C{H}$, the integer ${\rm sf}(A_0,A_1)$, called the spectral flow from $A_0$ to $A_1$, is introduced in \cite[Sect. 7]{AtiyPatoSing76}. In particular, for an odd theta-summable Fredholm module $(\C{H},\slashed{D})$ on $\C{A}$, and a unitary $g \in U_N(\C{A})$, the spectral flow defines a pairing
\[
K^1(\C{A})\times K_1(\C{A}) \lra \B{Z}, \qquad \langle D, g \rangle := {\rm sf}(D,g^{-1}Dg)
\]
between the $K$-theory and the $K$-homology of $\C{A}$. Furthermore, it is shown in \cite[Sect. 2]{Getz93} that
\[
{\rm sf}(D,g^{-1}Dg) = \frac{1}{\sqrt{\pi}}\,\displaystyle\int_0^1\,{\rm Tr}(\overset{.}{\slashed{D}_u}e^{-\slashed{D}^2_u})du,
\]
where $\slashed{D}_u:=(1-u)\slashed{D}+ug^{-1}\slashed{D}g$, and $\overset{.}{\slashed{D}_u}:=g^{-1}[\slashed{D},g]$.

\subsection{JLO-cocycle revisited} 

Let $\B{C}[\B{R}]$ be the polynomial commutative and cocommutative
Hopf algebra generated by $e(t)$ where $t\in\B{R}$ subject to the
relations
\[ e(t)e(s) = e(t+s) \quad\text{ and }\quad \Delta(e(t)) = e(t)\otimes
e(t) \]
for every $t,s\in\B{R}$.  We set $1 = e(0)$.  This is the group ring
$\B{C}[\B{R}]$ of the group $(\B{R},+,0)$.  

Let $\C{Z}$ be the 2-dimensional coalgebra $\left<1,D\right>$,
where $1$ is group-like and $D$ is primitive.
Consider the operator \( \vp_{n}\colon \C{A}^{\otimes n+1}\to \Hom(\C{Z}^{\otimes n+1}\times\Delta^{n},\B{C}) \)
defined by
\[ \vp_{n}(a_{0}\otimes\cdots a_{n})(D^{\epsilon_0}\otimes\cdots\otimes D^{\epsilon_n}|t_0,\ldots,t_n)
   = Str(a_0e^{-t_0\slashed{D} ^{\epsilon_0}}[\slashed{D} ,a_1] e^{-t_1\slashed{D} ^{\epsilon_1}}\cdots [\slashed{D} ,a_n]e^{-t_n\slashed{D} ^{\epsilon_n}}),
\]
and let us restrict the arguments in \(\B{C}[\B{R}]^{\otimes n+1}\) to
\(\Delta^{n}\).  In the next step we split the space of continuous
functionals \(\Hom(\C{Z}^{\otimes n+1}\times \Delta^{n},\B{C})\) as
\(\Hom(\C{Z}^{\otimes n+1},\B{C})\cotimes\Hom(\Delta^{n},\B{C})\).
Using the restriction function
\({\rm rest}: \gamma_{n}\mapsto \gamma_{n}|_{1\otimes\cdots\otimes 1} \) on
$\Hom(\C{Z}^{\otimes n+1},\B{C})$, and the integral on
\(\Hom(\Delta^{n},\B{C})\), we get a trace on the product.  The
JLO-cocycle is the composition
\[\begin{CD}
\C{A}^{\otimes n+1} @>{\vp_n}>> 
\Hom(\C{Z}^{\otimes n+1}\times\Delta^{n},\B{C}) @>{\cong}>> 
\Hom(\C{Z}^{\otimes n+1},\B{C})\cotimes \Hom(\Delta^{n},\B{C}) @>{{\rm rest}\,\otimes \int}>>
\B{C}.
\end{CD}
\]
In the next subsections we are going to construct a new cocycle in the
image of a Connes-Moscovici type characteristic map which behaves the
same way as the JLO-cocycle in terms of the Chern pairing with the
K-theory.

\subsection{Characteristic map}\label{CM-CharacteristicMap}

In this subsection we are going to construct a characteristic map from
the asymptotic cyclic cohomology of the simplex we calculated in
Subsection~\ref{AsymptoticSimplex} to the entire cohomology of the
algebra $\C{A}$.  Then we are going to observe that the image of the
cocycle \eqref{cocycle-D} yields the index of the theta-summable
Fredholm module $(\C{H},\slashed{D})$ after paired with the K-theory.

\begin{proposition}\label{CharMap}
  Let $\C{A}$ be a $\C{B}$-module algebra for a Hopf-algebra $\C{B}$
  with the modular pair in involution $(\ve,1)$. Assume also that
  $\C{B}$ acts on $e(t)$ trivially, for all $t\geq 0$, and the
  supertrace ${\rm Str}$ is invariant under the $\C{B}$-action. There
  is a characteristic homomorphism of cocyclic objects given by
  $\chi\colon {\rm Diag}^\bullet(C^\bullet_{\C{B}}(\C{B};\epsilon,1)\ot \Delta^\bullet)
  \lra C^\bullet(\C{A})$
  \begin{align*}
    \chi(h^1 & \odots h^n \mid t_1,\ldots, t_n) (a_0,\ldots, a_n) \\
    := & {\rm Str}\left(a_0e(t_1)h^1(a_1)e(t_2-t_1)\ldots e(t_n-t_{n-1})h^n(a_n)e(1-t_n)\right),
  \end{align*}
  where $C^\bullet_{\C{B}}(\C{B};\epsilon,1)$ is the standard Hopf cocyclic object
  associated to $\C{B}$ with coefficients in the MPI $(\epsilon,1)$.
\end{proposition}

\begin{proof}
We are going to see the compatibility of the characteristic homomorphism with the cocyclic structure. To begin with,
\begin{align*}
\chi(d_0(h^1 & \odots h^n)\mid \delta_0(t_1, \ldots, t_n)) (a_0,\ldots, a_{n+1})\\
 = & \chi(1 \ot h^1\odots h^n\mid 0, t_1, \ldots, t_n) (a_0,\ldots, a_{n+1})\\
 = & {\rm Str}\left(a_0a_1e(t_1)h^1(a_2)e(t_2-t_1)\cdots h^n(a_{n+1})e(1-t_n)\right) \\
 = & d_0\chi(h^1\odots h^n \mid t_1, \ldots, t_n) (a_0,\ldots, a_{n+1}).
\end{align*}
Similarly, for $1 \leq i \leq n$, we have
\footnotesize
\begin{align*}
 \chi(d_i(h^1 & \odots h^n) \mid \delta_i(t_0, \ldots, t_n)) (a_0,\ldots, a_{n+1}) \\
 = & \chi(h^1\odots h^i\ps{1} \ot h^i\ps{2} \odots h^n \mid t_1, \ldots, t_i, t_i, \ldots, t_n) (a_0,\ldots, a_{n+1})\\
 = & {\rm Str}\left(a_0e(t_1)h^1(a_1)e(t_2-t_1)\cdots e(t_i-t_{i-1})h^i\ps{1}(a_i)e(t_i-t_i)h^i\ps{2}(a_{i+1})e(t_{i+1}-t_i)\cdots h^n(a_{n+1})e(1-t_n)\right) \\
 = & {\rm Str}\left(a_0e(t_1)h^1(a_1)e(t_2-t_1)\cdots e(t_i-t_{i-1})h^i(a_ia_{i+1})e(t_{i+1}-t_i)\cdots h^n(a_{n+1})e(1-t_n)\right) \\
 = & d_i\chi(h^1\odots h^n \mid t_1, \ldots, t_n) (a_0,\ldots, a_{n+1}).
\end{align*}
\normalsize
For the last coface operator, we have
\begin{align*}
\chi(d_{n+1}(h^1 & \odots h^n) \mid \delta_{n+1}(t_1, \ldots, t_n)) (a_0,\ldots, a_{n+1}) \\
= & \chi(h^1\odots h^n \ot 1 \mid t_1, \ldots, t_n, 1) (a_0,\ldots, a_{n+1})\\
= & {\rm Str}\left(a_0e(t_1)h^1(a_1)e(t_2-t_1)\ldots h^n(a_n)e(t_n-t_{n-1})a_{n+1}\right) \\
= & {\rm Str}\left(a_{n+1}a_0e(t_1)h^1(a_1)e(t_2-t_1)\ldots h^n(a_n)e(t_n-t_{n+1})\right) \\
= & d_{n+1}\chi(h^1\odots h^n \mid t_1, \ldots, t_n) (a_0,\ldots, a_{n+1}).
\end{align*}
We proceed to the compatibility with the codegeneracies. For $0\leq j \leq n-1$,
\footnotesize
\begin{align*}
\chi(s_j(h^1 & \odots h^n) \mid \s_j(t_1, \ldots, t_n)) (a_0,\ldots, a_{n+1}) \\
 = & \chi(h^1\odots \ve(h^{j+1})\odots h^n \mid t_1, \ldots, t_j,t_{j+2}, \ldots, t_n) (a_0,\ldots, a_{n-1})\\
 = & {\rm Str}\left(a_0e(t_1)h^1(a_1)e(t_2-t_1)\cdots 
         h^j(a_j)e(t_{j+1}-t_j)h^{j+1}(1)e(t_{j+2}-t_{j+1})h^{j+2}(a_{j+1})\ldots h^n(a_{n-1})e(1-t_n)\right) \\
 = & \chi(h^1\odots h^n \mid t_1, \ldots, t_n) (a_0,\ldots, a_j, 1, \ldots, a_{n-1}) \\
 = & s_j\chi(h^1\odots h^n \mid t_1, \ldots, t_n) (a_0,\ldots, a_{n-1}).
\end{align*}
\normalsize
Finally, the compatibility with the cyclic operator goes as follows:
\begin{align*}
\chi(t_n(h^1 & \odots h^n) \mid \tau_n(t_1, \ldots, t_n)) (a_0,\ldots, a_n) \\
 = & \chi(S(h^1)\cdot (h^2\odots h^n \ot 1) \mid t_2-t_1, \ldots, t_n-t_1, 1-t_1) (a_0,\ldots, a_n) \\
 = & {\rm Str}\left(a_0e(t_2-t_1)S(h^1)\cdot (h^2(a_1)e(t_3-t_2)\cdots h^n(a_{n-1})e(1-t_n) 1(a_n)e(t_1))\right) \\
 = & {\rm Str}\left(h^1(a_0)e(t_2-t_1)h^2(a_1)e(t_3-t_2)\cdots h^n(a_{n-1})e(1-t_n) 1(a_n)e(t_1))\right) \\
 = & {\rm Str}\left(a_ne(t_1)h^1(a_0)e(t_2-t_1)h^2(a_1)e(t_3-t_2)\cdots h^n(a_{n-1})e(1-t_n))\right) \\
 = & \chi(h^1\odots h^n \mid t_1, \ldots, t_n) (a_n, a_0,\ldots, a_{n-1}) \\
 = & t_n\chi(h^1\odots h^n \mid t_1, \ldots, t_n) (a_0,\ldots, a_n),
\end{align*}
where we used the triviality of the $\C{B}$-action on $e(t)$, and the invariance of ${\rm Str}$ under the $\C{B}$-action.
\end{proof}

\subsection{Even index cocycle}

We note that, the morphism of cocyclic objects we defined in
Proposition~\ref{CharMap} induces a map on the asymptotic complexes
if $\C{B}$ is a Hopf algebra in $\Banach$:
\[ \chi: C^i_{\C{E}(x_n)}({\rm Diag}(C^\bullet_{\C{B}}(\C{B};\epsilon,1) \ot \Delta^\bullet))
\lra C^i_{\C{E}(x_n)}(C^\bullet(\C{A})). \]
On the other hand by Proposition~\ref{Diagonal} we have a cup product,
and therefore, we get
\[\xymatrix{
  C^i_{\C{E}(x_n)}(C^\bullet_{\C{B}}(\C{B};\epsilon,1)) \ot C^j_{\C{E}(1)}(\Delta^\bullet) \ar[r]^{\cup} 
  & C^{i+j}_{\C{E}(x_n)}({\rm Diag}(C^\bullet_{\C{B}}(\C{B};\epsilon,1)\ot \Delta^\bullet)) \ar[r]^{\hspace{1cm}\chi} 
  & C^{i+j}_{\C{E}(x_n)}(C^\bullet(\C{A})).
}\]
We are going to use this setup to define an index cocycle.

Now, let $\C{P}$ be the polynomial Hopf algebra $\B{C}[X]$, and let us define
\begin{equation}\label{FundamentalCocycle}
  {\bf I}^0 := {\bf 1} \in C^0(\C{P}). \quad\text{ and }\quad
  {\bf I}^r := \underset{r-many}{\underbrace{{\bf I}\, \cup \, \ldots \, \cup\, {\bf I}}} \in C^{2r}(\C{P}), 
\end{equation}
for any $r\geq 1$ where ${\bf I}:=1\otimes 1 \in C^2(\C{P})$.
For this Hopf algebra we record the following.

\begin{proposition}\label{prop-embedding}
  We have a morphism of differential graded $\B{C}$-vector spaces of the form
  \begin{equation}\label{iota}
    \iota: C^*_{\C{E}(1)}(\Delta^\bullet) \lra C^0_{\C{E}(1)}(C^\bullet(\C{P})) \cotimes C^*_{\C{E}(1)}(\Delta^\bullet),
  \end{equation}
  given by
  \begin{align*}
      & ({\bf t}_{2n})_{n\geq 0} \mapsto \left(\sum_{r=0}^n \,\a_{r}\,({\bf I}^r \mid {\bf t}_{2n-2r})\right)_{n\geq 0}
        \quad\text{ and }\quad
        ({\bf t}_{2n+1})_{n\geq 0} \mapsto \left(\sum_{r=0}^n \,\a_{r}\,({\bf I}^r \mid {\bf t}_{2n+1-2r})\right)_{n\geq 0},
  \end{align*}
  where $(\alpha_r)_{r\geq 0}$ is a sequence of real numbers defined as
  \begin{align*}
    \alpha_{r} =
    \begin{cases}
      1 & \text{ if } r=0\\
      \frac{1}{(2r)!} - \frac{1}{(2r-2)!}& \text{ if } r>0.
    \end{cases}
  \end{align*}
\end{proposition}

\begin{proof}
  We note that since ${\bf I} = b(1)$ we have $b({\bf I})=0$, hence a commutative square of the form
  \[\xymatrix{
      ({\bf t}_{2n})_{n\geq 0} \ar@{|->}[d]_{b} \ar@{|->}[r]^{\iota\hspace{15mm}} 
      &  \left(\sum_{r=0}^n\,\a_{r}\, ({\bf I}^r \mid {\bf t}_{2n-2r})\right)_{n\geq 0} \ar@{|->}[d]^{\hb + \vb} \\
      (b{\bf t}_{2n})_{n\geq 0} \ar@{|->}[r]_{\iota\hspace{15mm}} 
      & \left(\sum_{r=0}^n\,\a_{r}\, ({\bf I}^r \mid b{\bf t}_{2n-2r})\right)_{n\geq 0}.
    }
  \]
Thus \eqref{iota} commutes with the Hochschild coboundary maps. Similarly, we have
\[\tau_2({\bf I}) = S(1)\cdot (1 \ot 1)  = {\bf I}.\]
Therefore, we also have $B_2({\bf I}) = 0$. As a result, we have another commutative square of the form
\[\xymatrix{
    ({\bf t}_{2n}) \ar@{|->}[d]_{B} \ar@{|->}[r]^{\iota\hspace{15mm}} 
    &  \left(\sum_{r=0}^n \,\a_{r}\,({\bf I}^r \mid {\bf t}_{2n-2r})\right)_{n\geq 0} \ar@{|->}[d]^{\hB + \vB} \\
    (B{\bf t}_{2n}) \ar@{|->}[r]_{\iota\hspace{15mm}} 
    & \left(\sum_{r=0}^n\,\a_{r}\, ({\bf I}^r \mid B{\bf t}_{2n-2r})\right)_{n\geq 0}.
  }
\]
This means \eqref{iota} also commutes with the Connes boundary maps.
\end{proof}

We now have the characteristic map
\[ \chi\circ \cup \circ \iota\colon 
  C^i_{\C{E}(1)}(\Delta^\bullet) 
  \lra 
  C^i_{\C{E}(1)}(C^\bullet(\C{A})).
\]
In fact, following the estimation given in \cite[Lemma 2.1]{GetzSzen89}, it is not difficult to see that
\[ \chi\circ \cup \circ \iota(\vp)
\in  C^i_\ve(C^\bullet(\C{A})) \subset C^i_{\C{E}(1)}(C^\bullet(\C{A})).
\]

In the next step, we are going to consider the image  of the cocycle \eqref{cocycle-D} under this characteristic map. Explicitly, we are going to observe that the pairing between the image of the cocycle \eqref{cocycle-D} and the (topological) K-theory of the algebra $\C{A}$ yields the index of the Fredholm module $(\C{H},\slashed{D})$ up to a non-zero constant. The pairing between the entire cyclic cohomology and the K-theory is established in \cite[Thm. 8]{Conn88} which we recall below.

\begin{theorem}\label{thm-pairing}
  Let $\phi:= (\phi_{2n})_{n\geq 0} \in C^0_\ve(\C{A})$, and
  \[ F_\phi(x) := \sum_{n\geq 0} \,\frac{(-1)^n(2n)!}{n!}\,\phi_{2n}(x,x,\ldots, x) \]
  be the corresponding entire function on $M_\infty(\C{A})$. Then the additive map
  \[ \langle \phi,\,\,\rangle : K_0(\C{A}) \lra \B{C}, \qquad [p] \mapsto \langle \phi,[p]\rangle := F_\phi(p) \]
  depends only the class $[\phi] \in HP_\ve^0(\C{A})$.
\end{theorem}

We now present the main result of this subsection in the discussion below.

  Let us introduce the element ${\bf v} = (v_n)_{n\in\B{N}} \in C^0_{\C{E}(1)}(\Delta^\bullet)$,
  \begin{equation}
    v_n:=\delta_0^{2n}(*) = (\underset{2n\text{-many}}{\underbrace{0,\ldots,0}}).
  \end{equation}
  We do not claim that ${\bf v}$ is a cocycle, but it will be useful  nonetheless.

\begin{lemma}\label{lemma-ev-p}
For any idempotent
  $[p] \in K_0(\C{A})$ that acts on the Hilbert space $\C{H}$ as a
  self-adjoint operator, and $[\slashed{D},p]=0$, we have
  \[\langle \chi\circ \cup \circ \iota({\bf v}),[p]\rangle =  {\sum_{n\in\B{N}}}\frac{(-1)^n}{n!}\,{\rm ind}(\slashed{D}_p).\]
\end{lemma}

\begin{proof}
  From the definition of the characteristic homomorphisms, and Theorem \ref{thm-pairing}, it follows that
  \begin{align}\label{ev-p}
    \langle \chi\circ \cup \circ \iota({\bf v}),[p]\rangle 
    = &  {\sum_{n\in\B{N}}}\frac{(-1)^n(2n)!}{n!}\,\left(\sum_{r=0}^n\a_{r}\right)\,{\rm Str}\left(pe(0)pe(0)\cdots e(0)pe(1)\right) \nonumber\\
    = &  {\sum_{n\in\B{N}}}\frac{(-1)^n}{n!}\,{\rm Str}(pe^{-\slashed{D}^2}).
  \end{align}
  On the other hand, 
  \[ {\rm Str}(pe^{-\slashed{D}^2}) = {\rm Str}(pe^{-\slashed{D}^2}p) = {\rm Str}(pe^{-\slashed{D}_p^2}p), \]
  where $\slashed{D}_p = p\slashed{D}p$, and thus the claim follows from the McKean-Singer formula \cite[Lemma 3.1]{GetzSzen89}.
\end{proof}

We are ready to state our main result.

\begin{theorem}\label{EvenIndexTheorem}
  Let $(\C{H},\slashed{D})$ be a Fredholm module over $\C{A}$, and
  $\vp \in C^0_{\C{E}(1)}(\Delta^\bullet)$ be the
  $\C{E}(1)$-asymptotic cyclic cocycle given by
  \eqref{cocycle-D}. Furthermore, let
  $\phi \in C^0_\ve(\C{A})$ be given by
  $\phi = \chi \circ \cup \circ \iota (\vp)$, and
  $\slashed{D}_p := p\slashed{D}p$. Then we have
  \[
    \langle [\phi], [p] \rangle = \left(\sum_{n\in\B{N}} \,\frac{n+1}{2^n(n!)^2}\right){\rm ind}(\slashed{D}_p).
  \]
\end{theorem}

\begin{proof}
Along the lines of \cite{GetzSzen89} we may choose the idempotent $[p] \in K_0(\C{A})$ as in Lemma \ref{lemma-ev-p}. Then using \eqref{cocycle-D}, \eqref{ev-p}, Lemma~\ref{lemma-ev-p} and Theorem~\ref{thm-pairing}, we get
  \begin{align*}
    \langle \phi,[p]\rangle
    = & \langle \chi\circ \cup \circ \iota(\vp),[p]\rangle \\
    = &  \sum_{n\in\B{N}} \,\frac{(-1)^n(2n)!}{n!}\,\frac{(-1)^n}{2^nn!}\,(n+1)\,\left(\sum_{r=0}^n\a_{r}\right){\rm ind}(\slashed{D}_p) \\
    = &  \left(\sum_{n\in\B{N}} \,\frac{n+1}{2^n(n!)^2}\right){\rm ind}(\slashed{D}_p).
  \end{align*}
  as we wanted to show.
\end{proof}

\subsection{Odd index cocycle}\label{OddChernPairing}

Given an odd theta-summable Fredholm module $(\C{H},\slashed{D})$ over an algebra $\C{A}$, we shall construct an odd asymptotic cocycle on $\C{A}$, using once again the cocycle \eqref{cocycle-D}, such that the pairing with a unitary $g \in K_1(\C{A})$ yields the spectral flow ${\rm sf}(\slashed{D}, g^{-1}\slashed{D}g)$.

Since \eqref{cocycle-D} is an even cocycle, while we need an odd one, we have to begin with the following embedding.

\begin{proposition}
  We have a morphism of differential graded $\B{C}$-vector spaces of the form
  \begin{equation}
    \eta: C^*_{\C{E}(1)}(\Delta^\bullet) \lra C^1_{\C{E}(1)}(C^\bullet(\C{P})) \cotimes C^*_{\C{E}(1)}(\Delta^\bullet),
  \end{equation}
  given by
  \begin{align*}
      & ({\bf t}_{2n})_{n\geq 0} \mapsto \left(X \mid {\bf t}_{2n}\right)_{n\geq 0}
        \quad\text{ and }\quad
        ({\bf t}_{2n+1})_{n\geq 0} \mapsto \left(X \mid {\bf t}_{2n+1}\right)_{n\geq 0}.
  \end{align*}
\end{proposition}

\begin{proof}
The claim follows, similar to Proposition \ref{prop-embedding}, from $X \in HC^1_{\C{E}(1)}(\C{P};\epsilon,1)$ being a cyclic 1-cocycle.
\end{proof}

On the next move, we transform this cocycle to an odd cocycle on the algebra via a characteristic homomorphism similar to the one given by Proposition \ref{CharMap}. 

\begin{proposition}\label{CharMap-II}
  Let $(\C{H},\slashed{D})$ be an odd theta-summable Fredholm module
  over an algebra $\C{A}$, let $g \in K_1(\C{A})$ be a unitary, and
  $\slashed{D}_u:=(1-u)\slashed{D}+ug^{-1}\slashed{D}g$.  Let also
  $\C{A}$ be a $\C{P}$-module algebra for a Hopf-algebra $\C{P}$ with
  the modular pair in involution $(\ve,1)$.  Assume further that the
  trace ${\rm Tr}$ is invariant under the $\C{P}$-action, and that
  $\C{P}$ acts on $e(t):=e^{-t\slashed{D}_u^2}$ trivially, for all
  $t\geq 0$. Then there is a characteristic homomorphism of cocyclic
  objects given by
  $\chi\colon {\rm Diag}^\bullet(C^\bullet_{\C{B}}(\C{B};\epsilon,1)\cotimes
  \Delta^\bullet) \lra C^\bullet(\C{A})$
  \begin{align*}
    \chi(h^1 & \odots h^n \mid t_1,\ldots, t_n) (a_0,\ldots, a_n) \\
    := & {\rm Tr}\left(a_0e(t_1)h^1(a_1)e(t_2-t_1)\ldots e(t_n-t_{n-1})h^n(a_n)e(1-t_n)\right).
  \end{align*}
\end{proposition}

Combining with Proposition \ref{Diagonal}, we obtain the odd asymptotic cocycle
\begin{equation}\label{odd-cocycle}
  (\chi\circ \cup \circ \eta)(\vp) \in C^1_\ve(\C{A}) \subset C^1_{\C{E}(1)}(\C{A}).
\end{equation}

We now need to pair \eqref{odd-cocycle} with the unitary
$g \in K_1(\C{A})$. To this end, we recall the odd analogue of the
pairing given by Theorem \ref{thm-pairing} from
\cite{CuntzQuillen:NonsingularityI}.  See also \cite{Getz93} and
\cite[Sect. 4.7]{Connes-book}.

\begin{theorem}\label{thm-pairing-odd}
  Given any $\phi:= (\phi_{2n+1})_{n\geq 0} \in C^1_\ve(\C{A})$, the additive map
  \[ \langle \phi,\,\,\rangle : K_1(\C{A}) \lra \B{C}, \qquad [g] \mapsto \langle \phi, g\rangle := \frac{1}{\sqrt{2\pi i}}\sum_{n=0}^\infty\,(-1)^n\,n!\,\phi_{2n+1}(g^{-1}, g, \ldots, g^{-1}, g) \]
  depends only on the class $[\phi] \in HP_\ve^1(\C{A})$.
\end{theorem}

We have now all the machinery we need.

\begin{theorem}\label{thm-spect-flow}
  Let $(\C{H},\slashed{D})$ be an odd theta-summable Fredholm module
  over an algebra $\C{A}$, $g \in K_1(\C{A})$ be a unitary, and
  $\slashed{D}_u:=(1-u)\slashed{D}+ug^{-1}\slashed{D}g$. Then,
  \[
    \langle (\chi\circ \cup \circ \eta)(\vp),\, g \rangle =
    \left(\frac{1}{\sqrt{2i}}\,\sum_{n=1}^\infty \,
      \frac{n+1}{2^n}\right)\,{\rm Tr}\left(g^{-1}X(g)e(1)\right).
  \]
\end{theorem}

\begin{proof}
  We observe that
  \begin{align*}
    \langle (\chi\circ \cup & \circ \eta)(\vp),\, g \rangle\\
    = & \frac{1}{\sqrt{2\pi i}}\,{\rm Tr}\left(g^{-1}X(g)e(1)\right)
        + \frac{1}{\sqrt{2\pi i}}\,\sum_{n=1}^\infty\sum_{r=0}^n\frac{(-1)^n}{2^nn!}\,(-1)^n\,n!\,{\rm Tr}\left(g^{-1}X(g)e(1)\right),
  \end{align*}
  since we have
  \begin{align*}
    {\rm Tr} & \left(g^{-1}e(0)X\ps{1}(g)\ldots e(0)X\ps{2(n-r)+1}(g)e(1)X\ps{2(n-r)+2}(g^{-1})e(0)\ldots X\ps{2n+1}(g)e(0)\right) \\
    = & {\rm Tr}\left(g^{-1}X(g)e(1)\right)
  \end{align*}
  for any $0 \leq r \leq n$.
\end{proof}

Finally, integrating on $u \in [0,1]$, 
\begin{eqnarray*}
  \displaystyle \int_0^1\, \langle \chi\circ \cup \circ \eta(\vp),\, g \rangle \,du
  & = & \frac{1}{\sqrt{2\pi i}}\,\sum_{n=1}^\infty \, \frac{n+1}{2^n}\,\int_0^1\,{\rm Tr}\left(\overset{.}{\slashed{D}_u}e^{-\slashed{D}_u^2}\right)\,du \\
  & = & \left(\frac{1}{\sqrt{2i}}\,\sum_{n=1}^\infty \, \frac{n+1}{2^n}\right)\, {\rm sf}(\slashed{D}, g^{-1}\slashed{D}g)
\end{eqnarray*}
we obtain the spectral flow ${\rm sf}(\slashed{D}, g^{-1}\slashed{D}g)$, up to a constant multiple.

\bibliographystyle{plain}
\bibliography{references}{}

\begin{thebibliography}{10}

\bibitem{AtiyPatoSing76}
M.~F. Atiyah, V.~K. Patodi, and I.~M. Singer.
\newblock Spectral asymmetry and {R}iemannian geometry. {III}.
\newblock {\em Math. Proc. Cambridge Philos. Soc.}, 79:71--99, 1976.

\bibitem{Conn88}
A.~Connes.
\newblock Entire cyclic cohomology of {B}anach algebras and characters of
  {$\theta$}-summable {F}redholm modules.
\newblock {\em $K$-Theory}, 1(6):519--548, 1988.

\bibitem{Connes-book}
A.~Connes.
\newblock {\em Noncommutative geometry}.
\newblock Academic Press Inc., 1994.

\bibitem{ConnesHigson:AsymptoticMorphisms}
A.~Connes and N.~Higson.
\newblock Déformations, morphismes asymptotiques et {$K$}-théorie bivariante.
\newblock {\em C. R. Acad. Sci. Paris Sér. I Math.}, 311(2):101--106, 1990.

\bibitem{ConnMosc98}
A.~Connes and H.~Moscovici.
\newblock Hopf algebras, cyclic cohomology and the transverse index theorem.
\newblock {\em Comm. Math. Phys.}, 198(1):199--246, 1998.

\bibitem{ConnMosc99}
A.~Connes and H.~Moscovici.
\newblock Cyclic cohomology and {H}opf algebras.
\newblock {\em Lett. Math. Phys.}, 48(1):97--108, 1999.
\newblock Moshé Flato (1937--1998).

\bibitem{CuntzQuillen:NonsingularityI}
J.~Cuntz and D.~Quillen.
\newblock Algebra extensions and nonsingularity.
\newblock 8(2):251--289, 1995.

\bibitem{Getz93}
E.~Getzler.
\newblock The odd {C}hern character in cyclic homology and spectral flow.
\newblock {\em Topology}, 32(3):489--507, 1993.

\bibitem{GetzJone93}
E.~Getzler and J.~D.~S. Jones.
\newblock The cyclic homology of crossed product algebras.
\newblock {\em J. Reine Angew. Math.}, 445:161--174, 1993.

\bibitem{GetzSzen89}
E.~Getzler and A.~Szenes.
\newblock On the {C}hern character of a theta-summable {F}redholm module.
\newblock {\em J. Funct. Anal.}, 84(2):343--357, 1989.

\bibitem{Higson:KKTheory}
N.~Higson.
\newblock Categories of fractions and excision in {$KK$}-theory.
\newblock {\em J. Pure Appl. Algebra}, 65(2):119--138, 1990.

\bibitem{JaffLesnOste88}
A.~Jaffe, A.~Lesniewski, and K.~Osterwalder.
\newblock Quantum {$K$}-theory. {I}. {T}he {C}hern character.
\newblock {\em Comm. Math. Phys.}, 118(1):1--14, 1988.

\bibitem{Kassel:1987}
C.~Kassel.
\newblock Cyclic homology, comodules, and mixed complexes.
\newblock {\em J. Algebra}, 107(1), 1987.

\bibitem{Khalkhali-thesis}
M.~Khalkhali.
\newblock {\em On the entire cyclic cohomology of {B}anach algebras}.
\newblock ProQuest LLC, Ann Arbor, MI, 1991.
\newblock Thesis (Ph.D.)--Dalhousie University (Canada).

\bibitem{KhalRang04}
M.~Khalkhali and B.~Rangipour.
\newblock On the generalized cyclic {E}ilenberg-{Z}ilber theorem.
\newblock {\em Canad. Math. Bull.}, 47(1):38--48, 2004.

\bibitem{KlimLesn93}
S.~Klimek and A.~Lesniewski.
\newblock A note on the entire cyclic cohomology of a finite-dimensional
  noncommutative space.
\newblock {\em Canad. Math. Bull.}, 36(4):449--457, 1993.

\bibitem{KRT:HopfCyclicCohomologyOfQuantumGroups}
J.~Kustermans, J.~Rognes, and L.~Tuset.
\newblock The {C}onnes-{M}oscovici approach to cyclic cohomology for compact
  quantum groups.
\newblock 26(2):101--137, 2002.

\bibitem{Loday-book}
J.~L. Loday.
\newblock {\em Cyclic homology}, volume 301 of {\em Grundlehren der
  Mathematischen Wissenschaften [Fundamental Principles of Mathematical
  Sciences]}.
\newblock Springer-Verlag, second edition, 1998.
\newblock Appendix E by Maria O. Ronco, Chapter 13 by the author in
  collaboration with Teimuraz Pirashvili.

\bibitem{Puschnigg:AsymptoticCohomology}
M.~Puschnigg.
\newblock {\em Asymptotic cyclic cohomology}, volume 1642 of {\em Lecture Notes
  in Mathematics}.
\newblock Springer-Verlag, Berlin, 1996.

\bibitem{Quillen:KTheory}
D.~Quillen.
\newblock Higher algebraic {$K$}-theory. {I}.
\newblock In {\em Algebraic $K$-theory, I: Higher $K$-theories (Proc. Conf.,
  Battelle Memorial Inst., Seattle, Wash., 1972)}, pages 85--147. Lecture Notes
  in Math.,Vol. 341. Springer, 1973.

\bibitem{Ryan-book}
R.~A. Ryan.
\newblock {\em Introduction to tensor products of {B}anach spaces}.
\newblock Springer Monographs in Mathematics. Springer-Verlag London, Ltd.,
  London, 2002.

\bibitem{Weib-book}
C.~A. Weibel.
\newblock {\em An introduction to homological algebra}, volume~38 of {\em
  Cambridge Studies in Advanced Mathematics}.
\newblock Cambridge University Press, 1994.

\end{thebibliography}

\end{document}